\documentclass[reqno]{amsart}
\usepackage[foot]{amsaddr}

\usepackage[hidelinks]{hyperref}

\usepackage{url}
\usepackage[nameinlink,capitalize]{cleveref}

\usepackage{amssymb}
\usepackage{amscd}
\usepackage{amsthm}
\usepackage{setspace}
\usepackage{enumerate}
\usepackage{graphicx}
\usepackage{mathtools}
\usepackage{tcolorbox} 
\usepackage{subcaption} 
\usepackage{siunitx}
\usepackage{tikz-cd}
\usepackage{bm,blkarray}
\numberwithin{equation}{section}
\usepackage{algorithm}
\usepackage{algpseudocode}

\usepackage{mathtools}
\usepackage[tableposition=top]{caption}
\usepackage{booktabs,dcolumn}

\renewcommand\P{\mathbb{P}}
\newcommand\E{\mathbb{E}}

\newcommand\R{\mathbb{R}}

\newcommand\N{\mathbb{N}}

\newcommand\Unif{\mathbf{Unif}}

\newcommand\C{\mathbb{C}}

\newcommand\dist{{\operatorname{dist}}}

\newcommand\eps{\varepsilon}

\newcommand{\norm}[1]{{\left\lVert #1 \right\rVert}}

\renewcommand\mod[1]{\ (\mathop{\rm mod}#1)}

\newcommand{\diag}{\mathop {\rm diag}}

\newcommand{\wass}{\mathcal W_1}
\newcommand{\tvdist}{\dist_{\textnormal{TV}}}
\newcommand{\sphere}[1]{\mathcal S^{#1}}

\usepackage{xcolor}      %
\usepackage{listings}    %

\definecolor{codegreen}{rgb}{0,0.6,0}
\definecolor{codegray}{rgb}{0.5,0.5,0.5}
\definecolor{codepurple}{rgb}{0.58,0,0.82}
\definecolor{white}{rgb}{1,1,1}

\lstdefinestyle{mystyle}{
    backgroundcolor=\color{white}, 
    commentstyle=\color{codegreen},
    keywordstyle=\color{magenta}, numberstyle=\tiny\color{codegray}, stringstyle=\color{codepurple}, basicstyle=\ttfamily\footnotesize, breakatwhitespace=false,  breaklines=true,                 
    captionpos=b,                    
    keepspaces=true,                 
    numbers=left,                    
    numbersep=5pt,                  
    showspaces=false,                
    showstringspaces=false,
    showtabs=false,                  
    tabsize=2,
    frame=single,
    rulecolor=\color{black},
    title=\lstname
}

\renewcommand{\mod}{\bmod}

\title[Spectral density estimation]{Spectral density estimation for normal matrices}

\author{Cameron Musco$^\dagger$}
\address{{\scriptsize $^\dagger$Manning College of Information and Computer Sciences, University of Massachusetts Amherst, Amherst, MA 01003}}

\email{cmusco@cs.umass.edu}
\author{Christopher Musco$^\ddagger$}
\address{\scriptsize $^\ddagger$Department of Computer Science and Engineering, New York University, New York, NY 10012}
\email{cmusco@nyu.edu}

\author{Rikhav Shah$^\ast$}
\address{\scriptsize $^\ast$Department of Mathematics, Massachusetts Institute of Technology, Cambridge, MA 02139.}
\email{rdshah@mit.edu}

\author{John Urschel$^\ast$}
\email{urschel@mit.edu}

\author{Nicholas West$^\ast$}
\email{npwest00@mit.edu}

\subjclass[2020]{Primary 65F15; Secondary 68Q17}
\keywords{Normal matrices, spectral density estimation, kernel polynomial method, algorithmic lower bounds, Wasserstein distance, randomized algorithms, matrix-vector query complexity}

\newtheorem{theorem}{Theorem}[section]

\newtheorem{lemma}[theorem]{Lemma}

\newtheorem{proposition}[theorem]{Proposition}

\newtheorem{remark}{Remark}

\usepackage{multicol}
\usepackage{mleftright}
\usepackage{enumitem}
\usepackage{caption}
 \usepackage[nameinlink,capitalize]{cleveref}

\usepackage{refcount}
\newcommand{\secthm}[1]{\texorpdfstring{\Cref{#1}}{Theorem \getrefnumber{#1}}}
\newcommand{\secprop}[1]{\texorpdfstring{\Cref{#1}}{Proposition \getrefnumber{#1}}}

\newcommand{\qfor}{\quad\text{for}\quad}

\newcommand{\normal}{{\mathcal N}}

\renewcommand{\Re}{\mathrm{Re}}
\renewcommand{\Im}{\mathrm{Im}}

\newcommand{\wrt}{\,\textnormal d}

\newcommand{\wt}{\widetilde}
\newcommand{\wh}{\widehat}

\newcommand{\abs}[1]{\mleft|#1\mright|}

\newcommand{\magn}[1]{\mleft\|#1\mright\|}

\newcommand{\pare}[1]{\mleft(#1\mright)}

\newcommand{\sqbrac}[1]{{\left[{#1}\right]}}
\newcommand{\set}[1]{{\left\{{#1}\right\}}}
\newcommand{\bmat}[1]{\begin{bmatrix}#1\end{bmatrix}}
\newcommand{\pmat}[1]{\begin{pmatrix}#1\end{pmatrix}}

\newcommand{\ceil}[1]{\mleft\lceil#1\mright\rceil}
\newcommand{\floor}[1]{\mleft\lfloor#1\mright\rfloor}

\newcommand{\spliteq}[2]{\begin{equation}#1\begin{split}#2\end{split}\end{equation}}
\newcommand{\eq}[1]{\begin{equation}{#1}\end{equation}}

\DeclareMathOperator*{\Cov}{Cov}

\DeclareMathOperator*{\tr}{tr}

\DeclareMathOperator{\lcm}{lcm}

\newcommand{\tth}{^{\textnormal{th}}}

\makeatletter
\newtheorem*{rep@theorem}{\rep@title}
\newcommand{\newreptheorem}[2]{%
\newenvironment{rep#1}[1]{%
 \def\rep@title{#2 \ref{##1}}%
 \begin{rep@theorem}}%
 {\end{rep@theorem}}}
\makeatother

\newreptheorem{theorem}{Theorem}
\newreptheorem{coro}{Corollary}

\AddToHook{env/lemma/begin}{\crefalias{theorem}{lemma}}
\AddToHook{env/corollary/begin}{\crefalias{theorem}{corollary}}
\AddToHook{env/proposition/begin}{\crefalias{theorem}{proposition}}

\crefname{theorem}{Theorem}{Theorems}
\crefname{lemma}{Lemma}{Lemmas}
\crefname{corollary}{Corollary}{Corollaries}
\crefname{equation}{}{}
\crefname{subsection}{subsection}{subsections}
\crefname{proposition}{Proposition}{Propositions}

\newcommand{\KPM}{\textnormal{KPM}}
\renewcommand{\epsilon}{\eps}

\renewcommand{\paragraph}[1]{
\vspace{2.5mm}\noindent\textbf{#1. }
}

  \usepackage{nth}
  \usepackage{intcalc}

  \newcommand{\cAAAI}[1]{AAAI\ Conference\ on\ Artificial (AAAI)}

\begin{document}

\begin{abstract}
The spectral density estimation problem asks for an algorithm that, given an $n\times n$ matrix $A$, outputs a probability measure that is a good approximation to the uniform distribution on the eigenvalues of $A$, called the spectral density of $A$. This paper considers the setting where $A$ is a large normal matrix that is accessible only through matrix-vector product queries. 
We provide an algorithm that makes just $m$ matrix-vector queries to $A$ and returns, with high probability, a measure within earth mover's distance $O(1/m+\log m/{\sqrt n})$ of the true spectral density of $A$. We provide a complementary lower bound that any algorithm producing an $\eps$-approximation to the true spectral density for large matrices must make $\Omega(1/\eps)$ matrix-vector queries. The lower bound holds even for the more restricted case of real symmetric input matrices. In combination with our upper bound, it shows that spectral density estimation is essentially no harder for complex normal matrices than for real symmetric matrices.
\end{abstract}

\maketitle

\section{Introduction}
Given a matrix $A \in \mathbb C^{n \times n}$, its \textit{spectral density} $\mu_A$ is the uniform distribution on its discrete multiset of $n$ complex eigenvalues.
Spectral density estimation (SDE) is the problem of producing, given access to $A$, a measure $\wh\mu_A$ that approximates $\mu_A$ in some metric; we will use the earth mover's distance (also called the Wasserstein distance, Kantorovich-Rubinstein distance, and optimal transport distance).

Spectral density estimation arises as a natural problem in a variety of contexts: in physics and quantum chemistry (where the spectral density is often referred to as the  ``density of states'') \cite{b0,b1}, in the study of the Hessians of neural nets during optimization \cite{b2}, and as a subroutine in large-scale eigenvalue computation \cite{b3}. 
In many of these contexts, the matrix $A$ is constrained to be real symmetric or Hermitian, and there are several standard references for spectral density estimation in these cases, e.g. \cite{b4,b1,b5}.
In this paper, we expand the scope of spectral density estimation to \textit{normal} matrices.
Normal matrices which are not Hermitian arise in several settings. The most important class is unitary matrices.
Other examples include circulant matrices, normal random matrix models in physics \cite{b6,b7}, normally regular digraphs \cite{b8}, and directed Cayley graphs where the set of generators is contained in its own normalizer.
Moreover, several matrix operations preserve normality: applying matrix functions, taking Kronecker products, adding commuting normal matrices, and taking the compound matrix (matrix of minors), to name a few.

The spectral density can in principle be computed to a high degree of precision by diagonalizing $A$, at cost $O(n^\omega)$ (or $O(n^3)$ in practice) \cite{b9}. For modern problems where $n$ is enormous or $A$ is given implicitly (e.g., as a black box), this is prohibitive. The earth mover's relaxation---only requiring that $\mu_A$ and $\wh\mu_A$ are close in earth mover's distance---opens the door to dramatically faster algorithms.

In particular, our algorithm is a matrix-vector query (matvec) algorithm.
These algorithms interact with $A$ only through queries of the form $v \mapsto Av$ (and possibly $A^*v$). The matvec complexity of a problem is the number of such queries needed to solve it; in the worst-case, each query may cost $O(n^2)$, and so we obtain a substantially faster algorithm than full diagonalization when the query complexity is $o(n^{\omega-2})$ (or just $o(n)$ in practice).
Important classes of matrix-vector query algorithms are Krylov subspace methods and randomized linear sketching methods, both of which are frequently used to approximate outlying eigenvalues of matrices and solve related spectral problems \cite{b10,b11,b12,b13,b14,b15}.

For Hermitian matrices, the well-known stochastic Lanczos quadrature (SLQ) and kernel polynomial method (KPM) both show that $m$ matrix-vector queries suffice to achieve an earth mover's approximation error of $O(1/m)$ to the spectral density of any Hermitian matrix when the matrix dimension $n$ is sufficiently large \cite{b16,b17}.
A restricted lower bound has also been shown: if the very small approximation error of $O(1/n)$ is desired, then $\Omega(n)$ queries are required \cite{b18,b19}.

Both SLQ and KPM use the fact that the spectrum of $A$ lies on the real line, and SLQ additionally uses the property that compressions of Hermitian matrices are themselves Hermitian. Neither of those hold in the general normal setting \cite{b20}. Because the spectrum of $A$ can be distributed throughout a 2D set in $\C$, one might expect more matrix-vector queries are needed to resolve it. A heuristic based on packing suggests the problem should be quadratically harder: at error $\epsilon$, one can fit $\Theta(1/\epsilon^2)$ well-separated eigenvalue clusters in a bounded region of the plane, versus $\Theta(1/\epsilon)$ on the line. Indeed, such a loss is inherent for related problems like numerical range estimation \cite{b21}.

Contrary to this intuition, we show that $m$ matvecs suffice to achieve approximation error $O(1/m)$ for large normal matrices $A$. In particular, normal SDE is no harder than Hermitian SDE in the matvec model.
We additionally show a substantially broader lower bound as compared to \cite{b18,b19}: for any $\eps>0$, if an algorithm produces an $\eps$ approximation for large matrices, then we show $\Omega(1/\eps)$ queries are required.

Before stating our results more formally, we review some standard notation. We use $\magn\cdot$ to denote the $\ell_2$ (Euclidean) norm  of vectors, and the spectral norm of a matrix. $\magn\cdot_p$ denotes the $\ell_p$ norm of vectors.
For a set $S$, $\magn{f}_S=\sup_{x\in S}\abs{f(s)}$ is the sup-norm of a function $f$ on $S$.
The $\ell_p$-Lipschitz constant of $f:\C^n\to\C$ is $\sup_{x\neq y}\frac{\abs{f(x)-f(y)}}{\magn{x-y}_p}$. 
The Dirac delta distribution $\delta(z)$ is defined by $\int f(z)\delta(z-c)\wrt z=f(c)$ for continuous functions $f$. $\wass$ is the earth mover's distance
defined for compactly supported probability measures $\mu$, $\nu$ as
$\wass(\mu,\nu)=\sup_f\pare{\int f(z)(\text d\mu(z)-\wrt\nu(z))}$,  where the supremum is over functions $f$ with $\ell_2$-Lipschitz constant bounded by 1. We note that there is an equivalent definition of $\wass(\cdot,\cdot)$ in terms of the optimal transport distance between $\mu$ and $\nu$. The above characterization, which will be easier to work with, is often known as Kantorovich-Rubinstein duality. We denote the spectral density of a matrix $A \in \mathbb{C}^{n \times n}$ as
\begin{align}\label{a0}
    \mu_A(z) \equiv \dfrac{1}{n} \sum_{j=1}^n \delta(z - \lambda_j),
\end{align}
where $\lambda_1, \ldots, \lambda_n$ are the eigenvalues of $A$.
Our main algorithmic result is the following.
\begin{theorem}\label{a1}
    Let \(A\in\C^{n\times n}\) be a normal matrix with spectral density $\mu_A$.
    Fix \(0 < \delta< 1\). \cref{a2} computes an approximating measure \(\widehat \mu_A\) satisfying, with probability \(1-\delta\),
    \eq{\label{a3}\dfrac{1}{\norm{A}}\wass(\mu_A, \widehat \mu_A) \le O\pare{\frac1m+\sqrt{\dfrac{\log(m/\delta) \log(m)}{n}}}.}
    The algorithm uses \(O(m)\) matrix-vector products with \(A\) and \(A^*\), \(O(mn)\) storage, and \(O(m^2n)\) additional running time. The probability density function $w(x,y)$ satisfying $w(x,y) \wrt x \wrt y = \wrt \widehat \mu_A(x,y)$ may be queried at a point $(x,y)$ in $O(m^2)$ time. Moreover, the measure \(\widehat \mu_A\) may be chosen to be supported on just \(O(m^2)\) atoms with an additional $O(m^3)$ run-time.
\end{theorem}

\begin{remark}
    Several variants of \Cref{a2} are described in \Cref{a4} for various purposes and applications, all still achieving \cref{a3} or the appropriate analog.
\end{remark}

\begin{remark}
In addition to applying to a broader class of matrices, \Cref{a2}
improves upon the theoretical run-times of the algorithms of \cite{b17,b19} for Hermitian SDE. Namely, the $m$ dependence is improved by avoiding an optimization problem used by those algorithms to enforce strict moment-matching not required in our analysis.
Additionally, simplifying \Cref{a2} to Hermitian matrices would allow us to output $\wh\mu_A$ supported on just $O(m)$ atoms in the Hermitian case, compared to $O(m^{1.5})$ atoms in the algorithm of \cite{b19}. We also note that the arguments of \cite{b19} for producing atomic measures do extend to the non-Hermitian / 2D setting, but this only gives a construction for $O(m^3)$ atoms compared to $O(m^2)$ of \Cref{a1}.
\end{remark}

\begin{figure}[t]
  \centering
    \includegraphics[trim={0 1.5cm 0 1cm},clip,width=1\linewidth]{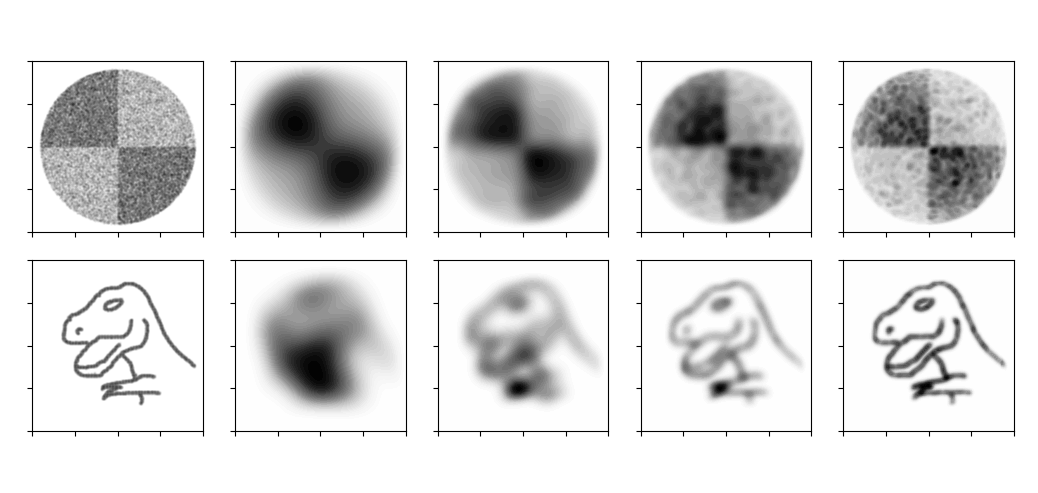} 
    \caption{Spectral density estimation for two example matrices.
    The leftmost column is a scatter plot of the eigenvalues of the matrices. The next columns left-to-right are the result of running \Cref{a2} for $m=16$, $32$, $64$, and $128$ respectively.
    The first row is with the 150,000$\times$150,000 matrix where 100,000 eigenvalues are uniformly distributed in the disk, and 50,000 additional eigenvalues are uniformly distributed in the intersection of the disk with the second and fourth quadrants.
    The second row is with the 27,078$\times$27,078 matrix
    which has its eigenvalues trace out a version of the ``Anscombosaurus'' \cite{b22,b23}. All plots depict the unit square $[-1,1]+i[-1,1]\subset\C$.
    The total time for generating the entire figure (including plotting) on a standard laptop took less than 2 seconds.
    }
    \label{a5}
\end{figure}

An alternative way to state \Cref{a1} is that to achieve a target error of $\eps\magn A$ in earth mover's distance for an $n \times n$ matrix with $n$ sufficiently large, \Cref{a2} need only make $O(1/\eps)$ queries. We complement this result with a matching lower bound. 

\begin{theorem}\label{a6}
There exist absolute constants $C, C''$ and an integer-valued function $N(\eps,s)$ satisfying $N(\eps,s)\le\pare{\frac3{8C\eps}}^s$ such that the following holds.
Fix any triplet of parameters $s \in\N$ with $s \geq 2$, $0 < \eps < (4Cs)^{-1}$, and $n\in N(\eps,s)\N\cap[C''s^2\eps^{-6},\infty)$. 
If a randomized, adaptive, matrix-vector query algorithm, on input any $n\times n$ real symmetric matrix $A$, outputs a distribution that is within $\eps \magn A$ in earth mover's distance of the spectral density of $A$ with probability more than ${1.01}s^{-1}$, then it must make $(32C\eps)^{-1}$ matrix-vector queries.
\end{theorem}

\begin{remark}
    The lower bound of \Cref{a6} is against algorithms which need only succeed for real symmetric inputs. Thinking of $\eps$ and $s$ as fixed constants, this matches our upper bound in the general normal case,  establishing that there is no gap between the complexities of the normal and real symmetric spectral density estimation problems.
\end{remark}

\begin{remark}
A weaker lower bound changes the order of the quantifiers on the parameters: suppose one wants an algorithm which for fixed $n$ can produce an $\eps$ approximation for all $\eps>0$, including $\eps$ going to 0. In this case, there are several existing lower bounds, e.g. \cite{b24,b19,b25} all prove the following: for each $n$, there exists an $\eps$ such that approximating the spectral density to $\eps$ earth mover's error requires $O(1/\eps)$ queries.
In particular, their constructions are for $\eps=\Theta(1/n)$. Stated differently, they show that to achieve $O(1/n)$ error, one must learn a constant fraction of the matrix. However, this is not typically the setting in which SDE is interesting.
The goal is to obtain fast algorithms which require very few queries relative to the matrix size and output a relatively coarse approximation.

    In the practical regime where $\eps$ is a small, fixed constant (or tends to 0 slowly), and $n$ is large, the previous lower bounds do not apply. \Cref{a6} addresses this gap: for $s=6$, it applies whenever $\eps=\Omega(1/n^{1/6})$. This resolves the complexity of spectral density estimation outside the window $1/n\ll \eps\ll1/n^{1/6}$.
\end{remark}

 \Cref{a6} follows from a more general statement, \Cref{a7}, which isolates the information-theoretic bottleneck at the core of spectral density estimation using the ideas of \cite{b26}. We concretely find that any algorithm making only $m$ matrix vector queries to real symmetric $A$ does not learn much unitarily-invariant information about $A$ beyond the values of $\tr(A^j)$ for $j=1,\ldots,4m$, which are just the first $4m$ moments of the spectral density of $A$. These moments are both necessary and sufficient to approximate the spectral density of $A$.

\subsection{Numerical experiments} Numerical experiments were performed in IEEE double precision in MATLAB R2025b on a MacBook Pro with an Apple M2 chip, 16GB of RAM, and a single computational thread. The code generating \Cref{a5} and \Cref{a8} is available upon request to the final author. All matrices were stored as sparse diagonal matrices, although the code can be modified to handle any matrix with a fast matrix-vector product routine.

\section{Algorithm and its analysis}\label{a9}

In this section, we derive \Cref{a2} for approximating the spectral density of a normal matrix and  discuss its theoretical properties, most notably by proving \Cref{a1}. We begin by introducing the algorithm itself: given matrix-vector access to \(A\) and \(A^*\), we explain how to compute mixed Chebyshev moments of a random surrogate spectral measure and then use these moments to form an approximation to $\mu_A$. After presenting the algorithm, we turn to its analysis via a multivariate version of the kernel polynomial method (KPM). Although the KPM is well known in the Hermitian setting \cite{b4}, rigorous error bounds for approximating probability distributions are more recent \cite{b17}, and the multivariate version we require here has received comparatively little attention, excepting the appendix of \cite{b19}. For this reason, we include a discussion of Jackson's theorems, their multivariate extension, and how they are deployed in our setting. We conclude the section with some algorithmic variants, and we illustrate numerical experiments in \Cref{a5} and \Cref{a8}.

\paragraph{Weighted spectral density}
A key object appearing in our analysis is the \textit{weighted} spectral density of a matrix $A$. If we have the eigendecomposition $A=\sum_{j=1}^n\lambda_jv_jv_j^*$, then the weighted spectral density with respect to a unit vector $b\in\C^n$ is the measure
\begin{equation}\label{a10}
\mu_{A,b}(z)\equiv\sum_{j=1}^n\delta(z-\lambda_j)\abs{b^*{v_j}}^2.
\end{equation}
When $b=\frac1{\sqrt n}\pare{v_1+\cdots+v_n}$, this coincides with the unweighted spectral measure $\mu_{A,b}=\mu_A$ defined in \eqref{a0}. The analysis of our algorithm consists of two steps: (1) show that we can compute an approximation to $\mu_{A,b}$ for any $b$ of our choice (without ever explicitly forming $\mu_{A,b}$), and (2) show that $\mu_{A,b}$ is an approximation to $\mu_A$ with high probability over a Haar random selection of $b$.

\paragraph{Chebyshev polynomials and moments}
Playing a role in both steps mentioned above are the Chebyshev polynomials and Chebyshev moments. 
 We let \(T_j(x)\) denote the \(j\tth\) Chebyshev polynomial of the first kind,  defined by the three-term recurrence,
\spliteq{\label{a11}}{
T_0(x) &= 1, \quad T_1(x) = x, \\
T_{n+1}(x) &= 2x\,T_n(x) - T_{n-1}(x), \quad n \ge 1.}
The Chebyshev polynomials are orthogonal on the interval \([-1, 1]\) with respect to the weight function \(w(x) = (1-x^2)^{-1/2}\): 
\begin{equation}\label{a12}
    \int_{-1}^1 \dfrac{T_j(x)T_k(x)}{\sqrt{1-x^2}}\wrt x = \begin{cases}
        0 & j \neq k \\
        \pi & j = k = 0 \\
        \pi/2 & j = k \geq 1 
    \end{cases}.
\end{equation}
We write \(\wt{T}_0(x) = 1/\sqrt{\pi}\) and \(\wt{T}_j(x) = \sqrt{2/\pi} \cdot T_j(x)\) for the \textit{normalized} Chebyshev polynomials, which are an orthonormal family.
For a measure $\mu$ supported on the square $[-1,1]^2$,
define the $(j,k)$ mixed Chebyshev moment as
\[\Gamma_{jk} = \int_{-1}^1 \int_{-1}^1 \wt T_j(x)\wt T_k(y) \wrt\mu(x, y).\]
That is, \(\Gamma_{jk}\) is obtained by integrating the bivariate polynomial \(\wt T_j(x)\wt T_k(y)\) over the unit square with respect to \(\mu\). By identifying $z=x+iy\in\C$, this definition applies to measures $\mu$ on the unit square $[-1,1]+i[-1,1]$ in $\C$ as well.

In our application to spectral density estimation, we  assume \(0 <\norm{A}_2 \leq 1\), or that an upper bound $L$ satisfying $\norm{A} \leq L \leq C \norm{A}$ has been supplied, where $C$ is a universal constant. Given $L$, one may apply the algorithm to $\wt A = A/L$ with $\|\wt A\| \leq 1$. One may then rescale the output \(\widehat\mu\) to \(\widehat \mu(zL)\), which incurs a factor of $L = O(\norm{A})$ error in earth mover's distance.\footnote{This is a standard assumption in research on spectral density estimation. Obtaining a constant-factor relative error approximation to the norm of a matrix can be done using a number of matrix-vector products polylogarithmic in the dimension $n$ and is discussed for example in \cite{b27,b28}.}

\subsection{Chebyshev moments of weighted spectral measure}

Our first goal is to compute the mixed Chebyshev moments $\Gamma_{jk}$ for the weighted spectral measure $\mu_{A,b}$.
The key observation is that
that the Toeplitz decomposition \(A = H + iK\), i.e. \(H = (A+A^*)/2\) and \(K = (A-A^*)/2i\) of a normal matrix allows us to compute these moments without actually forming $\mu_{A,b}$. %
\begin{lemma}\label{a13}
    Let \(A = H+iK\) be the Toeplitz decomposition of the normal matrix \(A\) so that \(H = (A+A^*)/2\) and \(K = (A-A^*)/2i\). Then for any fixed polynomials \(p, q\) it holds that
    \begin{equation}
        \int_{-1}^1\int_{-1}^1\overline p(x)q(y)\wrt\mu_{A,b}(x,y) = b^* p(H)^*q(K)b
    \end{equation}
    where $\bar p$ denotes the polynomial with coefficients that are the conjugates of the coefficients of $p$.
\end{lemma}
\begin{proof}
Since $A$ is normal, by the spectral theorem we may write its eigendecomposition as $A=U\Lambda U^*$ where $U$ is unitary. From the definitions $H=(A+A^*)/2$ and $K=(A-A^*)/2i$, we have that $H$ and $K$ have the corresponding eigendecompositions
    \[
        H = U \Re(\Lambda)U^*, \quad K = U \Im(\Lambda)U^*.
    \]
    Then we observe that the quadratic form simplifies considerably: 
    \begin{align*}
        b^* \overline{p}(H)q(K)b & = b^*U\overline{p}(\operatorname{Re}(\Lambda))U^*Uq(\operatorname{Im}(\Lambda))U^*b \\
        & = (p(\operatorname{Re}(\Lambda))U^*b)^*(q(\operatorname{Im}(\Lambda))U^*b) \\ 
        & = \sum_{j=1}^n\overline{p}(\operatorname{Re}(\lambda_j))q(\operatorname{Im}(\lambda_j)) |(U^*b)_j|^2 \\
        & = \int_{-1}^1 \int_{-1}^1 \overline{p}(x)q(y) \wrt\mu_{A,b}(x, y).
    \end{align*}
    The first equality follows from the definition of matrix functions that if \(B = VDV^{-1}\) then \(f(B) = Vf(D)V^{-1}\), the second and third equalities rewrite the result as an inner product, and the final equality rewrites the discrete sum as an integral over the weighted spectral measure \(\mu_{A,b}\) as defined in \cref{a10}. 
\end{proof}

If we collect the mixed Chebyshev moments of \(\mu_{A,b}\) up to degree $m$ into an \((m+1)\)-by-\((m+1)\) matrix \(\Gamma\), then \Cref{a13} provides the factorization
\begin{equation}\label{a14}
    \Gamma = X^*Y
\end{equation}
where 
\[
    X = \bmat{\wt T_0(H)b& \wt T_1(H)b&\cdots& \wt T_m(H)b}
\]
and
\[
    Y= \bmat{\wt T_0(K)b&\wt T_1(K)b&\cdots&\wt T_m(K)b}.
\]
The matrices \(X, Y\) can be computed using the three-term recurrence of the Chebyshev polynomials, costing only \(m\) matrix-vector products with each of \(H\) and \(K\). If \(A\) and \(A^*\) are only available implicitly, then one can simulate the application of \(H\) via \(Hx = (Ax+A^*x)/2\), and similarly for \(K\). The resulting procedure is summarized in \Cref{a15}, which as a result requires \(O(m)\) matrix-vector products with \(H\) and \(K\) or \(O(m)\) matrix-vector products with \(A\) and \(A^*\). It also requires \(O(mn)\) storage for \(X\) and \(Y\) and \(O(m^2n)\) additional operations for the formation of \(\Gamma = X^*Y\). 

\begin{algorithm}[H]
    \caption{\textsc{ChebMoments\((A, b, m)\)} \\
    Computes mixed Chebyshev moments of \(\mu_{A, b}\) up to maximal degree \(m\) using three-term recurrence.}\label{a15}
    \begin{algorithmic}[1] %
        \Require $A\in\C^{n\times n}$, $\magn A\le1$, $b \in \C^{n}$, $m \geq 2$.
        \Ensure $(\Gamma)_{jk} = \int \wt T_j(x) \wt T_k(y) \wrt\mu_{A,b}(x, y)$ are the mixed Chebyshev moments of \(\mu_{A,b}\) of orders \(0 \leq j, k \leq m\).       
        \State $H\gets ({A+A^*})/2$
        \State $K\gets ({A-A^*})/{2i}$
        \State $x_0 \gets b/\sqrt{\pi}, \quad x_1 \gets \sqrt{2}Hx_0, \quad x_2 \gets 2Hx_1 - \sqrt{2}x_0$
        \State $y_0 \gets b/\sqrt{\pi},\quad  y_1 \gets \sqrt{2}Ky_0, \quad y_2 \gets 2 Ky_1 - \sqrt{2}y_0$
        \For{$j\gets 2$ to $m-1$}
            \State $x_{j+1} \gets2Hx_j-x_{j-1}$
            \State $y_{j+1} \gets2Ky_j-y_{j-1}$
        \EndFor
        \State $X \gets [x_0,  x_1,  x_2, \cdots,  x_m]$ 
        \State $Y \gets [y_0, y_1, y_2, \cdots, y_m]$
        \State $\Gamma \gets X^*Y $
        \State \Return $\Gamma$
    \end{algorithmic}
\end{algorithm}

\subsection{Approximating the weighted spectral measure. }
The previous subsection gave an algorithm for computing the mixed Chebyshev moments $\Gamma_{jk}$ of $\mu_{A,b}$. We now show how to use a multivariate version of the kernel polynomial method (KPM) \cite{b4} to convert these moments to an approximation to $\mu_{A,b}$.\footnote{This multivariate generalization of the KPM was described briefly but not analyzed in \cite{b4}.}
Our approximation will be
\begin{equation}\label{a16}
    \wrt\mu_{A,b}(x,y) \approx \wrt\mu_{\KPM}(x,y) \equiv \dfrac{p_m(x,y)}{\sqrt{1-x^2}\sqrt{1-y^2}}\, \wrt x\, \wrt y,
\end{equation}
where
\[
    p_m(x,y) = \sum_{j=0}^m\sum_{k=0}^m \Gamma_{jk} \rho_j\rho_k \wt{T}_j(x)\wt{T}_k(y).
\]
The \(\wt{T}_j\) are the normalized Chebyshev polynomials and the \(\rho_j\) are fixed constants, sometimes called \textit{damping coefficients}, which we specify momentarily. We summarize the resulting procedure in \Cref{a17}, which uses \Cref{a15} as a subroutine for computing the Chebyshev moments. The central theoretical result of this subsection is the following theorem, whose proof is found at the end of this subsection: 
\begin{theorem}\label{a18}
    Let \(\mu\) be a probability distribution on \([-1, 1]^2\) and let \(\mu_{\KPM}\) be defined as in \cref{a16} using the damping coefficients in \Cref{a19}. Then 
    \begin{equation}
        \wass(\mu, \mu_{\KPM}) \le \frac{12}{m}.
    \end{equation}
\end{theorem}
To understand $\mu_{\KPM}$, we first present the usual single-variable KPM formulation and discuss its generalization to multiple variables. The central object is the polynomial kernel $K_m$,
\begin{equation}\label{a20}
    K_m(x,r) \equiv \sum_{j=0}^m\rho_j \wt{T}_j(x) \wt{T}_j(r).
\end{equation}
One associates with \(K_m\) a linear operator (which we will also denote $K_m$) mapping bounded functions $f$ supported on \([-1, 1]\) to polynomials of degree at most \(m\). Application of this operator is performed via the formula
\begin{equation}\label{a21}
    (K_mf)(x) = \int_{-1}^1 \dfrac{K_m(x,r)f(r)}{\sqrt{1-r^2}}\wrt r.
\end{equation}
If $g(x)$ is a probability density function on the interval $[-1, 1]$ then the KPM approximation to $g(x)$ is the measure with density
\eq{\label{a22}\wt g(x)=
\frac1{\sqrt{1-x^2}}\int_{-1}^1K_m(x,r)g(r)\wrt r}
which can be interpreted as the composition of three operations on $g$: pointwise multiplication by $\sqrt{1-x^2}$,  application of the kernel $K_m$, and division by $\sqrt{1-x^2}$.
Importantly, this multiplication by the weight function $\sqrt{1-x^2}$ exactly cancels the weight function from the definition of the kernel application \cref{a21}. By expanding the kernel definition in the resulting approximation, one sees that $\wt g$ can be computed from just the knowledge of the Chebyshev moments of $g$, \(\int_{-1}^1T_j(r)g(r)\wrt r.\)

In order to ensure that $\wt g$ approximates $g$ and that $\wt g$ remains a probability density, the damping coefficients $\rho_j$ must be selected carefully.
One choice is to pick $\rho_j$ to be the optimum of the following linear program: minimize a quantity called the squared kernel resolution
\[Q(K_m) \equiv \int_{-1}^1 \int_{-1}^1 \dfrac{K_m(x,r)(x-r)^2}{\sqrt{1-x^2}\sqrt{1-r^2}}\wrt x \wrt r,\]
which rewards $K_m$ for being close to the identity operation, subject to
\begin{equation}\label{a23}
    \int_{-1}^1\dfrac{K_m(x,r)}{\sqrt{1-r^2}} \wrt r = 1, \quad K_m(x,r) \geq 0 \textrm{ for } x,r \in [-1, 1],
\end{equation}
which ensure $\wt g$ is nonnegative and integrates to 1, i.e. is a probability density function.
Those optimal values of $\rho_j$ are obtained in, e.g., \cite{b4,b29} and are given by
\begin{equation}\label{a19}
    \rho_j = \frac{(m+2-j) \cos\mleft(\dfrac{j\pi}{m+2}\mright) + \sin\pare{\dfrac{j \pi}{m+2}}\cot \mleft(\dfrac{\pi}{m+2}\mright)}{m+2}, \quad j = 0, 1, \cdots, m.
\end{equation}
A different derivation, which emphasizes trigonometric polynomials rather than algebraic polynomials, for these same coefficients was also obtained by Rivlin \cite{b30}.

KPM posits that $\wt g$ is close to $g$ in the earth mover's distance. To understand why, we summarize several results from \cite[Section 1.1.2]{b30}, which control how $K_m$ acts on Lipschitz test functions. When reading that reference, one should note that the coefficients in
\cite[Page 20]{b30} are not presented in the form of \Cref{a19}, but they are equivalent.\footnote{Make the substitutions $x = \cos \theta$ and $r = \cos \phi$ into our definition of $K_m$ and use the property of Chebyshev polynomials that $T_j(\cos \theta) = \cos (j \theta)$.}

\begin{lemma}\label{a24} Suppose $f$ is an $\ell$-Lipschitz function on $[-1, 1]$, and let the damping coefficients $\rho_0$, $\rho_1$, $\ldots$, $\rho_m$ be selected such that $K_m$ is positive and normalized. Then 
\begin{equation}
    \norm{f -K_mf}_{[-1, 1]} \leq \dfrac{\ell}{m}\left( 1+\dfrac{m \pi}{\sqrt{2}}\sqrt{1-\rho_1}\right).
\end{equation}
In particular, with the choice in \Cref{a19}, one obtains 
\eq{\label{a25}
    \norm{f - K_m f}_{[-1, 1]} \leq \dfrac{6\ell}{m}.}
\end{lemma}
\begin{remark}
    Results similar to \Cref{a24} of this kind are often called ``Jackson-type theorems,'' after Dunham Jackson, who showed that algebraic or trigonometric polynomials can be used to approximate Lipschitz functions with a sup-norm error of $O(1/m)$.
    Although the kernel $K_m$ that we use here arising from the choice \Cref{a19} is sometimes called ``Jackson's kernel'' in the literature, we note that Jackson's original construction as given in \cite{b31} is not exactly equivalent. Still, Jackson proves that the kernel he introduced in \cite{b31} obtains an analogous result, with lower and upper bounds on the relevant constant. The kernel of \cite{b31} is also analyzed in \cite[Appendix C]{b17}. 
\end{remark}
\begin{remark}
    Both Jackson's original kernel and that proposed by \cite{b30} were  analyzed in more detail concurrent with this work in \cite{b32}. The constant $6$ in \cref{a25} may be improved: in \cite{b32}, exact estimates were obtained for the kernel arising from the choice \Cref{a19}, for which the constant is $2.4306\dots +o_m(1)$, and for Jackson's original kernel, for which the constant is $2.6476\dots + o_m(1)$. Consequently, the choice to use \cref{a19} is slightly better than the use of Jackson's original kernel. 
\end{remark}

\Cref{a24} directly allows one to control the earth mover's distance between the measures with densities $g$ and $\wt g$:
fix a 1-Lipschitz test function $f$ and observe
\eq{\label{a26}\int_{-1}^1f(x)\wt g(x)\wrt x
=
\int_{-1}^1\int_{-1}^1\frac{f(x)K_m(x,r)g(r)}{\sqrt{1-x^2}}\wrt r\wrt x
=\int_{-1}^1 (K_mf)(r)g(r)\wrt r.}
Since \Cref{a24} implies $f$ and $K_mf$ are close, this means that the integrations of $f$ against $g$ and $\wt g$ are close, which is precisely what is required to show that $g$ and $\wt g$ are close in the earth mover's distance.
From \eqref{a26} we can directly bound:
\begin{align*}
    \wass(g,\wt g) = \sup_{\norm{f}_{\text{Lip}} \le 1}  \int_{-1}^1 f(z)(g(z) - \wt g(z)) d x = \sup_{\norm{f}_{\text{Lip}} \le 1} \int_{-1}^1 (f(z) - K_m f(z))g(z) d x \le \frac{6}{m},
\end{align*}
where the last inequality follows from \eqref{a25}.

One final observation of single-variable KPM before generalizing to multiple variables is that $K_m$ is a contraction in the following sense. 
\begin{lemma}\label{a27}
Let \(K_m\) be the polynomial kernel defined by \Cref{a20}, \Cref{a21}, and \Cref{a19}. \(K_m\) has the following contraction property:
    \begin{equation}
        \norm{K_m f}_{[-1, 1]} \leq \norm{f}_{[-1,1]}.
    \end{equation}
\end{lemma}
\begin{proof}
Simply expand the definition to see that
    \begin{align*}
    \norm{K_m f}_{[-1,1]} & \leq \max_{x \in [-1, 1]}\int 
    \dfrac{|K_m(x,r)f(r)|}{\sqrt{1-r^2}} \wrt r  \\
    & \leq \max_{y \in[-1,1]} |f(y)| \max_{x \in[-1, 1]}\int_{-1}^1 \dfrac{K_m(x,r)}{\sqrt{1-r^2}} \wrt r \\
    & = \norm{f}_{[-1,1]}.
    \end{align*}
\end{proof}

The preceding discussion for the univariate KPM generalizes in a straightforward way to multivariate problems in \(d \geq 1\) variables on the hypercube \([-1,1]^d\). In the following we state multivariate analogues of \Cref{a24} and \cref{a26}, and we consequently bound the KPM approximation error in earth mover's distance in \Cref{a18}.\footnote{While multivariate extensions of Jackson's theorem have been known for a long time (c.f. \cite{b33,b34}), they do not seem to be well-known and we provide a simple proof of \Cref{a28} using known properties of the one-dimensional \(K_m\).}
Define for a multivariate function \(f(x_1, x_2, \ldots, x_d)\) on \([-1,1]^d\) the component Lipchitz constants 
\[
\mathrm{Lip}_j(f)
\equiv \sup_{\substack{x_k\in[-1,1]\\ k\neq j}}
    \ \sup_{\substack{u,v\in[-1,1]\\ u\neq v}}
    \dfrac{\bigl|f(x_1,\dots,x_{j-1},u,x_{j+1}\dots,x_d)-f(x_1,\dots,x_{j-1},v,x_{j+1},\dots,x_d)\bigr|}{|u-v|}.
\]
$\mathrm{Lip}_j(f)$ is the maximal Lipschitz constant of \(f\) when viewed as a function of \(x_j\) with all other variables fixed. 

The one-dimensional kernel $K_m$ defined by \cref{a21} extends to a $d$ dimensional analog via the tensor product, i.e., let \(\otimes^d K_m = K_m \otimes K_m \otimes \cdots \otimes K_m \) be the integral operator defined by applying the one-dimensional Jackson operator separately in each coordinate: 
\eq{\label{a29}(\otimes^d K_m f)(x_1, \cdots, x_d) = \int_{[-1,1]^d} f(r_1,\cdots, r_d)\prod_{j=1}^d\left( \dfrac{K_m(x_j, r_j)}{\sqrt{1-r_j^2}}\right) \wrt r_1 \cdots \wrt r_d.}
Then we have the following multivariate extension of \Cref{a24}.
\begin{theorem}[Multivariate Jackson]\label{a28}
    Let \(f(x_1, x_2, \cdots, x_d)\) be a continuous function on \([-1, 1]^d\) and suppose there exist positive constants \(\ell_1, \cdots, \ell_d\) such that \(\mathrm{Lip}_{j}(f) \leq \ell_j\) for \(j = 1, 2, \cdots, d\). Then 
    \begin{equation}
        \norm{f - (\otimes^d K_m)f}_{[-1,1]^d} \leq \frac{6}{m}\sum_{j=1}^d \ell_j,
    \end{equation}
\end{theorem}
\begin{proof}
    We prove first \(d = 2\) for simplicity, and then describe the general proof afterward. We use \((I \otimes K_m)\) to indicate the linear integral operator that applies Jackson's one-dimensional kernel on the second variable:
    \[
        ((I \otimes K_m)f)(x,y) = \int_{-1}^1f(x, s) \dfrac{K_m(y,s)}{\sqrt{1-s^2}}\wrt s .
    \]
    The operator \((K_m \otimes I)\) is defined similarly as operating on the first variable. Under this notation, we may factor the iterated integrals via \((K_m \otimes K_m) = (I \otimes K_m)(K_m \otimes I) = (K_m \otimes I)(I \otimes K_m)\). Then observe
    \begin{align*}
        \norm{f - (K_m \otimes K_m)f}_{[-1,1]^2} & = \norm{f - (I \otimes K_m)f + (I \otimes K_m)(f - (K_m \otimes I)f)}_{[-1,1]^2} \\
        & \leq \norm{f - (I \otimes K_m)f}_{[-1,1]^2} + \norm{(I \otimes K_m)(f- (K_m \otimes I)f)}_{[-1,1]^2} \\
        & \leq \norm{f - (I \otimes K_m)f}_{[-1,1]^2} + \norm{f- (K_m \otimes I)f}_{[-1,1]^2} \\
        & \leq \dfrac{6 \ell_2}{m} + \dfrac{6 \ell_1}{m}.
    \end{align*}
    The first equality comes from adding and subtracting \((I \otimes K_m)f\). The inequality on the second line is an application of the triangle inequality. For the third line, let \(x\) be fixed and define the one-dimensional slice \(y \mapsto h_x(y) = (f - (K_m \otimes I)f)(x, y)\), so that we can apply \Cref{a27} by observing 
    \[
        \norm{(I \otimes K_m)(f- (K_m \otimes I)f)}_{[-1,1]^2} = \sup_{x \in [-1, 1]} \norm{K_mh_x} \leq \sup_{x \in [-1,1]} \norm{h_x}_{[-1,1]} = \norm{f - (K_m \otimes I)f}_{[-1,1]^{2}}.
    \]
    Finally, the fourth line is a consequence of applying the one-dimensional \Cref{a24} to the one-dimensional slices analogous to the argument applied to \(h_x\) above. More precisely, for fixed \(x\), observe that \(y \mapsto f_x(y) \equiv f(x,y)\) is an \(\ell_2\)-Lipschitz function of \(y\), so that 
    \[
        \norm{f - (I \otimes K_m)f}_{[-1,1]^2} = \sup_{x \in [-1,1]} \norm{f_x-K_mf_x}_{[-1,1]} \leq \dfrac{6\ell_2}{m}.
    \]
    The same argument may be applied to the analogously defined \(f_y\) to control the second term in the third line.
    
    The \(d > 2\) case follows from a similar argument. If we let \(P_j = (\otimes^{j}K_m)\otimes(\otimes^{d-j}I)\) be the operator which applies \(K_m\) in the first \(j\) coordinates, we apply the triangle inequality to 
    \[
    f-(\otimes^{d}K_m)f = \sum_{j=1}^{d} P_{j-1}f - P_{j}f,
    \]
    whose \(d\) terms lead to the factor of \(\sum_{j=1}^d \ell_j\) in the result. 
\end{proof}

\begin{remark}If \(f\) is \(\ell\)-Lipschitz with respect to any $p$-norm then \(\sum_j \mathrm{Lip}_j(f) \leq d \ell\).
\end{remark}

\begin{remark}
    Analogous results can be obtained by using a different kernel $K_m$ corresponding to a different choice of the $\rho_j$, such as Jackson's original construction \cite{b31}. 
\end{remark}

The calculation in \Cref{a26} immediately generalizes to the multivariate setting. For simplicity, we only state it for $d=2$ variables, which is all we need for complex measures.

\begin{lemma}\label{a30}
    Let \(\mu\) be a probability distribution on \([-1,1]^2\). Define further the approximation 
    \[
        \wrt \mu_\KPM(x, y) \equiv \dfrac{p_m(x,y)}{\sqrt{1-x^2} \sqrt{1-y^2}} \wrt x \wrt y
    \]
    where 
    \[
        p_m(x, y) = \sum_{i=0}^m \sum_{j=0}^m \rho_i \rho_j\Gamma_{ij} \wt{T}_i(x) \wt{T}_j(y), \quad \Gamma_{ij} = \int_{-1}^1 \int_{-1}^1 \wt{T}_i(x) \wt{T}_j(y) \wrt \mu(x,y).
    \]
    Then \(\mu_{\KPM}\) is a probability distribution and
    \[ 
        \int_{[-1,1]^2} f(x, y) \wrt \mu_{\KPM}(x, y)= \int_{[-1,1]^2}((K_m\otimes K_m)f)(x,y))\wrt \mu(x, y).
    \]
\end{lemma}

\begin{remark}
    We observe that \Cref{a30} shows that the ansatz \cref{a16} constructs a measure $\mu_{\KPM}$ whose action on a test function $f$ reproduces the action of the true measure $\mu$ on the test function after polynomial smoothing, $K_mf$. This is independent of the choice of the $\rho_j$.  
\end{remark}
We may now prove \Cref{a18}.
\begin{proof}[Proof of \Cref{a18}]
   Let $\mathcal{F}$ denote the class of 1-Lipschitz functions with respect to the Euclidean norm on $[-1, 1]^2$. Then applying \Cref{a30} and \Cref{a28} we obtain
    \begin{align*}
        \wass(\mu, \mu_{\KPM}) & = \sup_{f \in \mathcal{F}} \left( \int_{-1}^1\int_{-1}^1f(x,y)\wrt \mu(x,y) - \int_{-1}^1\int_{-1}^1f(x,y)\wrt \mu_{\KPM}(x,y)\right) \\
        & = \sup_{f \in \mathcal{F}} \left( \int_{-1}^1\int_{-1}^1\left(f - (K_m \otimes K_m)f\right)(x,y)\wrt \mu(x,y)\right) \\
        & \leq \dfrac{12}{m}.
    \end{align*}
\end{proof}

To be explicit about the algorithmic importance of \Cref{a18}, it shows that one may compute an approximation to a measure with just knowledge of its Chebyshev moments. The previous section showed how to compute those moments of $\mu_{A,b}$, and so we have an algorithm, \Cref{a17}, for computing an approximation $\widehat\mu_{A,b}$ of it.
\begin{algorithm}[H]
    \caption{\textsc{WSpecDens\((A, b, m)\)} \\
    Estimates weighted spectral density \(\widehat \mu_{A,b}\) using \Cref{a15} and Jackson's coefficients \Cref{a19}.}\label{a17}
    \begin{algorithmic}[1] %
        \Require $A\in\C^{n\times n}$, $\magn A\le1$, $b\in\C^n$, $\norm b=1$.
        \Ensure $\widehat \mu_{A,b}$ approximating $\mu_{A,b}$ with guarantee from \Cref{a18}.
        \State $\Gamma \gets \textsc{ChebMoments}(A, b, m)$
        \State $p_{\KPM}(x, y) \gets \sum_{j=0}^m\sum_{k=0}^m \rho_j \rho_k \Gamma_{jk}\wt{T}_j(x)\wt{T}_k(y)$
        \State \Return $\dfrac{p_{\KPM}(x,y)}{\sqrt{1-x^2}\sqrt{1-y^2}}\wrt x\wrt y$.
    \end{algorithmic}
\end{algorithm}

\subsection{Selection of weight vector.}
For any fixed vector $b$, the previous two subsections show how to compute an approximation of the weighted spectral density $\mu_{A,b}$. Our goal, however, is to approximate the unweighted spectral density $\mu_A$. Observe that if $b$ is a random vector sampled from an isotropic distribution, then
\[\E \mu_{A,b}=\mu_A.\]
This section argues that $\mu_{A,b}$ exhibits sufficient concentration when $b$ is sampled uniformly at random from the unit sphere. The final algorithm will then be \Cref{a2}.
\begin{algorithm}[H]
    \caption{\textsc{SpecDens\((A, m)\)} \\
    Estimates spectral density \(\widehat \mu_{A}\) using \Cref{a17}.}\label{a2}
    \begin{algorithmic}[1] %
        \Require $A\in\C^{n\times n}$, $\magn A\le1$.
        \Ensure $\widehat \mu_{A}$ approximating $\mu_{A}$ with guarantee from \Cref{a1}.
        \State Sample \(b \sim \Unif(\sphere{2n-1})\)
        \State \Return $\textsc{WSpecDens}(A,b,m)$.
    \end{algorithmic}
\end{algorithm}

Our analysis is in two parts, both appealing to known results. Claim \textbf{I} is that the Chebyshev moments of $\mu_{A,b}$ are highly concentrated around the corresponding moments of $\mu_A$. Claim \textbf{II} is that this suffices to imply $\mu_{A,b}$ is close to $\mu_A$ in earth mover's distance with high probability.

Claim \textbf{I} is a consequence of the following concentration result for Lipschitz functions on the sphere. We identify the unit sphere living in $\C^n$ with the unit sphere $\sphere{2n-1}$ living in $\R^{2n}$.
\begin{lemma}[{\cite[Theorem 10.3.1]{b35}}]\label{a31}
    Suppose \(g:\sphere{2n-1} \to \R\) has $\ell_2$ Lipschitz constant bounded by $L$. Let \(b \sim \Unif(\sphere{2n-1})\). Then 
    \[
        \Pr\pare{\abs{g(b) - \E[g(b)]} \geq r} \le 2e^{-(2n-1)r^2/(2L^2)}.
    \]
\end{lemma}
 
 We can use this result to bound the fluctuations in the Chebyshev moments of \(\mu_{A,b}\) around its expectation \(\mu_A\).

\begin{proposition}\label{a32}
    Let \(\Gamma_{jk}\) and \(\Gamma_{jk}'\) be the mixed Chebyshev moments of \(\mu_A\) and \(\mu_{A,b}\), respectively, for \(0 \leq j,k \leq m\). Let \(b \sim \Unif(\sphere{2n-1})\). Then for \(0 < \delta < 1\),
    \[
        \P\left(\max_{(j,k) \neq 0}|\Gamma_{jk}-\Gamma_{jk}'| \geq \sqrt{\dfrac{8}{2n-1} \log\left(\dfrac{2(m+1)^2-2}{\delta}\right)} \right) \leq \delta.
    \]
\end{proposition}
\begin{proof}
    For any \(f : \C \to \R\), and \(b \sim \Unif(\sphere{2n-1})\), we have 
    \begin{equation*}
        \int f(z) \wrt \mu_{A}(z) = \dfrac{1}{n}\tr(f(A)), \quad \int f(z) \wrt \mu_{A,b}(z) = b^*f(A)b.
    \end{equation*}
    Moreover, $\E bb^*=\frac1nI$ so
    \[
        \E [b^*f(A)b] = \E [\tr(f(A)bb^*)]
        = \dfrac{1}{n}\tr(f(A)).
    \]
    Observe also that \(g(b)\) defined by \(b \mapsto g(b)= b^*f(A)b\in\R\) is a Lipschitz function with Lipschitz constant bounded by \(2 \norm{f(A)}_2\) since
    \[
        |b^*f(A)b-c^*f(A)c| = |b^*f(A)(b-c)-(c-b)^*f(A)c| \leq 2 \norm{f(A)}\norm{b-c}.
    \]
    Thus, according to \Cref{a31}, 
    \begin{equation}\label{a33}
        \P\left(\left|b^*f(A)b - \dfrac{1}{n}\tr(f(A))\right| \geq t\right) \leq 2 e^{-(2n-1)t^2/(8 \norm{f(A)}^2)}.
    \end{equation}
    Let us now specialize to the case \(f_{ij}(z) = \wt{T}_{i}(x)\wt{T}_{j}(y)\). Since  $A = U\Lambda U^*$ is unitarily diagonalizable and following the notation of \Cref{a13}, we have
    \[f_{ij}(A) = Uf_{ij}(\Lambda)U^* = \pare{U\wt{T}_i(\Re (\Lambda)) U^*}\pare{U\wt{T}_j(\Im (\Lambda))U^*} = \wt{T}_i(H) \wt{T}_j(K).\]
    Moreover, under this notation,
    \[
        \Gamma_{ij}-\Gamma_{ij}' = \dfrac{1}{n}\tr(f_{ij}(A))-b^*f_{ij}(A)b,
    \]
    to which we may apply \Cref{a33} with a union bound over the \((m+1)^2-1\) polynomials. Note also that \(\norm{f_{ij}(A)} \leq 1\), so for  each \(0 < \delta < 1\) we may set $t=\sqrt{\dfrac{8}{2n-1} \log\left(\dfrac{2(m+1)^2-2}{\delta}\right)}$ for the final result.
\end{proof}

Claim \textbf{II} is proven in \cite{b19}. They show that if two measures have close Chebyshev moments, they must be close in earth mover's distance. Although the result is stated for $d$-dimensions in \cite{b19}, we state it only for 2-dimensions here. 
\begin{lemma}[{\cite[Theorem 32]{b19}}]\label{a34}
    Suppose \(\mu, \nu\) are probability distributions on \([-1, 1]^2\) with normalized Chebyshev moments \(\Gamma_{kj}\) and \(\Gamma_{kj}'\), respectively. If the distributions' normalized Chebyshev moments satisfy the bound 
    \begin{equation}\label{a35}
        \sum_{\substack{0 \leq k,j \leq m \\ (k,j) \neq (0, 0)}}\dfrac{1}{k^2+j^2} (\Gamma_{kj}-\Gamma_{kj}')^2 \leq K^2,
    \end{equation}
    for some constant $K > 0$,
    then one may bound the earth mover's distance 
    \[
        \wass(\mu, \nu) \leq \dfrac{4C}{m} + \pi K,
    \]
    where \(\pi/2 \leq C \leq 3\).\footnote{We have modified the result of \cite{b19} by using \(\pi/2 \leq C \leq 3\), as derived by Jackson in \cite{b31}, rather than the number 18 cited in Fact 12 of \cite{b19}. This improves the constant in \cite{b19} in the univariate case from \(c \leq36\) to \( c \leq 6\). A similar gain may be obtained by replacing the 18 in Fact 3.2 of \cite{b17}, since both of these works rely on Jackson's construction in \cite{b31}. A concurrent work \cite{b32} also provides the sharp estimate for $C = 12 \log(2) / \pi + o_m(1)$.}
\end{lemma}

In order to combine \Cref{a32,a34} to imply a bound on $\wass(\mu_A,\mu_{A,b})$, we must bound the sum over the terms of the form \((k^2+j^2)^{-1}\). For this, we introduce the following lemma. 
\begin{lemma}\label{a36}
    One may bound 
    \[
        \sum_{\substack{0 \leq k,j \leq m \\ (k,j) \neq (0, 0)}}\dfrac{1}{k^2+j^2} \leq 5 + \frac{\pi}{2} \log m.
    \]
\end{lemma}
\begin{proof}
We have
\[
    \sum_{\substack{0 \leq k,j \leq m \\ (k,j) \neq (0, 0)}}\dfrac{1}{k^2+j^2} =2 \sum_{k=1}^m\dfrac{1}{k^2} + \sum_{i=1}^m \sum_{j=1}^m \dfrac{1}{i^2+j^2}.
\]
The first term is bounded by $2\zeta(2)=\pi^2/3$ and the second term is bounded by
\begin{align*}
    \sum_{k=1}^m \sum_{j=1}^m \dfrac{1}{k^2+j^2} \le \sum_{k=1}^m \int_{0}^\infty \frac{1}{k^2+x^2} d x = \sum_{k=1}^m \frac{\pi}{2k} \le \frac{\pi}{2}(1+\log m).
\end{align*}
Lastly, $\pi^2/3+\pi/2<5$.
\end{proof}

\subsection{Proof of \secthm{a1}}
We now collect the results of the previous two subsections into a proof of our main result \Cref{a1}, which bounds the error in \Cref{a2}.

\begin{proof}[Proof of \Cref{a1}]
Denote the output of \Cref{a2} by \(\widehat \mu_A\), which is the KPM approximation to the random measure \(\mu_{A,b}\). Then observe that with probability \(1-\delta\) we have that 
\[
    \wass(\mu_A, \widehat\mu_{A}) \leq \wass(\mu_{A}, \mu_{A,b}) + \wass(\mu_{A,b}, \widehat\mu_A) \leq O\left(\dfrac{1}{m}\right)+O\left(\sqrt{\dfrac{\log(m) \log(m/\delta)}{n}}\right).
\]
Observe that the second term $\wass(\mu_{A,b}, \widehat \mu_A)$ is bounded by \Cref{a18}, contributing to the $O(1/m)$ term on the right-hand-side. 
The first term $\wass(\mu_A, \mu_{A,b})$ is bounded as follows. Apply H{\"o}lder's inequality to obtain 
\begin{align*}
    \sum_{\substack{0 \leq k,j \leq m \\ (k,j) \neq (0, 0)}}\dfrac{1}{k^2+j^2} (\Gamma_{kj}-\Gamma_{kj}')^2 & \leq \max_{(k,j) \neq (0, 0)}|\Gamma_{kj}-\Gamma'_{kj}|^2 \sum_{\substack{0 \leq k,j \leq m \\ (k,j) \neq (0, 0)}}\dfrac{1}{k^2+j^2} \\
    & \leq \pare{\dfrac{8}{2n-1} \log\left(\dfrac{2(m+1)^2-2}{\delta}\right)} \pare{5 + \frac{\pi}{2} \log m},
\end{align*}
where the final line holds with probability $1-\delta$ due to \Cref{a32} applied to the maximum over the $|\Gamma_{ij}-\Gamma'_{ij}|^2$ together with \Cref{a36} applied to the sum. Plugging this high-probability bound for $K^2$ into 
\Cref{a34} contributes an additional $O(1/m)$ term together with the stated $O(\sqrt{\log(m/\delta)n^{-1}})$ term. 

Moreover, the measure \(\widehat\mu_A\) may be (optionally) discretized using Gauss-Chebyshev quadrature on \((m+1)^2\) points. That is, let \(x_1, \cdots, x_{m+1}\) and \(w_1, \cdots, w_{m+1}\) be the nodes and weights of the Gauss quadrature associated with the Chebyshev measure. Recall that Gauss quadrature on \(k\) nodes is exact for polynomials of degree \(2k-1\). Then consider the discrete measure 
\begin{equation}\label{a37}
    \widehat \mu_{A, \textrm{discrete}} \coloneqq \sum_{i=1}^{m+1}\sum_{j=1}^{m+1} w_i w_j p_m(x_i, x_j) \delta_{(x_i, x_j)},
\end{equation}
where $p_m$ is as in \Cref{a16}.
The exactness of Gauss quadrature ensures 
\[
    \int_{[-1,1]^2} T_i(x)T_j(y) (\wrt \widehat \mu_A - \wrt \widehat \mu_{A, \textrm{discrete}}) = 0
\]
for all \(i, j \leq m\). When combined with \Cref{a34}, this implies that \(\wass(\widehat \mu_A, \widehat \mu_{A, \textrm{discrete}}) \leq 4C/m\), which merely alters the constant in the bound.

Finally, we discuss run-time considerations. Observe that \Cref{a15} requires $m$ matrix-vector products with $H$ and $K$, which can be formed explicitly at the start of the algorithm, or executed via implicit access to $A$ and $A^*$ via $x \mapsto Hx = (Ax+A^*x)/2$, $x \mapsto Kx = (Ax-A^*x)/2i$.  Thus, the algorithm need only compute $4m$ matrix-vector products with $A$ and $A^*$. The storage requirements are $O(nm)$ for $X, Y \in \mathbb{C}^{n \times (m+1)}$, and the $O(m^2n)$ additional running time comes from the $O(m^2)$ inner-products for vectors of length $n$ in the formation of $\Gamma = X^*Y$. 

Now suppose we wish to evaluate $p_\KPM$ on a tensor product grid $\mathcal{X} \times \mathcal{Y}  = \{(x,y): x\in \mathcal{X}, y \in \mathcal{Y}\}$, where $|\mathcal{X}|=|\mathcal{Y}|=k$. Let $T_\mathcal{X}$ be the $k \times (m+1)$ matrix whose $i$th row is given by $[\wt{T}_0(x_i), \wt{T}_1(x_i), \cdots, \wt{T}_{m}(x_i)]$. Evaluating the Chebyshev polynomials by their three-term recurrence requires $O(m)$ operations per row, so the construction of $T_\mathcal{X}$ requires $O(km)$ time. Define and construct $T_\mathcal{Y}$ similarly. Then if $P_{ij} = p_{\KPM}(x_i, y_j)$, we have  
\[
    P = T_\mathcal{X} D_\rho \Gamma D_\rho T_\mathcal{Y}^T,
\]
where $D_\rho = \diag(\rho_0, \rho_1, \cdots, \rho_m)$ is a diagonal matrix of the damping coefficients and $\Gamma$ is the mixed moment matrix. The matrix $M = D_\rho \Gamma D_\rho$ can be computed in $O(m^2)$ time, then $M' = T_\mathcal{X} M$ requires $O(km^2)$ time, and forming $P = M'T_\mathcal{Y}^T$ requires $O(k^2 m)$ time. This means that $p_{\KPM}$ may be queried at a single point ($k = 1$) in $O(m^2)$ time, and at a tensor product grid of $(m+1)^2$ points ($k = (m+1)$) in $O(m^3)$ time. Since $w(x,y) = p_{\KPM}(x,y)(1-x^2)^{-1/2}(1-y^2)^{-1/2}$, we may then evaluate $w(x,y)$ at a single point in $O(m^2)$ time, or we may form $\widehat \mu_{A, \textrm{discrete}}$, which requires the evaluation of $p_{\KPM}$ on the tensor product grid with $\mathcal{X}=\mathcal{Y}$ the Chebyshev nodes of order $m+1$, in $O(m^3)$ time. 

\end{proof}

\subsection{Algorithmic variants}\label{a4} We include here some comments about algorithmic variants that may be useful.

\begin{enumerate}
    \item It is possible to consider an \textit{adjoint-free} version of \Cref{a2} that avoids the use of the adjoint \(A^*\), which is an important feature in some applications, particularly PDEs \cite{b36}. One way to accomplish this is to use $m-1$ matrix-vector products to form the explicit Krylov matrix $\mathcal{K}_m = [b, Ab, \cdots, A^{m-1}b].$
    In analogy to \Cref{a14}, the Gram matrix $\mathcal{K}_m^*\mathcal{K}_m$ contains the ``polyanalytic'' mixed moments corresponding to polynomials in $z$ and $\bar z$. Since there is an invertible linear change of coordinates from \(z=x+iy, \bar{z}=x-iy\) into \(x, y\), one could in principle pre-compute the change of basis coefficients from \(z^k\bar z^l\) into the Chebyshev polynomials \(\wt{T}_i(x) \wt{T}_j(y)\). This may require high precision depending on the conditioning of the change of basis map between polynomial moments. %
   
    \item It is possible to consider a low-memory variant of \Cref{a2}. We could compute \(X^*Y\) in \cref{a15} one row at a time, which requires \(O(m^2)\) matvecs but only \(O(n)\) storage. 
    \item Say the spectrum of \(A\) is somewhat oblong, e.g., the range of real parts is decently larger than the range of imaginary parts.
    Then rather than computing \(m\) matvecs in both the \(x\) and \(y\) coordinates, we could decouple the \(x\) and \(y\) directions into \(m_1\) and \(m_2\) matvecs, so that \(X^*Y \in \mathbb{R}^{m_1 \times m_2}\).
    \item If the \(\wt{O}(n^{-1/2})\) error from \Cref{a1} is too high, one could apply \Cref{a2} to the block matrix \(A \otimes I_s \in \mathbb{C}^{ns \times ns}\), where $I_s$ is the $s\times s$ identity matrix and $\otimes$ is the Kronecker product. This would decrease the error due to randomness to \(O((ns)^{-1/2})\). Note that this is equivalent to replacing $\mu_{A,b}$ with a weighted average of $s$ copies of $\mu_{A,b}$ with independent $b$. By picking $s=\lceil m^{2/3}/n^{1/3}\rceil$, this improves the overall error to $\wt O(1/(nm)^{1/3})$ when $m=\Omega(\sqrt n)$.
    \item Because of the decoupling of the $x$- and $y$-coordinate directions, the resulting measure seems to exhibit axis-aligned artifacts. One can remove these artifacts by  ``rotating'' the input, i.e. applying the algorithm to $e^{i\theta }A$ for a few angles $\theta$, rotating back the returned measures, and averaging. That is, use the estimate \(\frac1{\abs\Theta}\sum_{\theta\in\Theta}\widehat\mu_{e^{i\theta}A}\circ(z\mapsto e^{i\theta}z)\) for some set of angles $\Theta$. It may also be possible to work with other polynomial bases, such as disk polynomials \cite{b37}.
\end{enumerate}

\begin{figure}[t]
  \centering
    \includegraphics[width=0.7\linewidth]{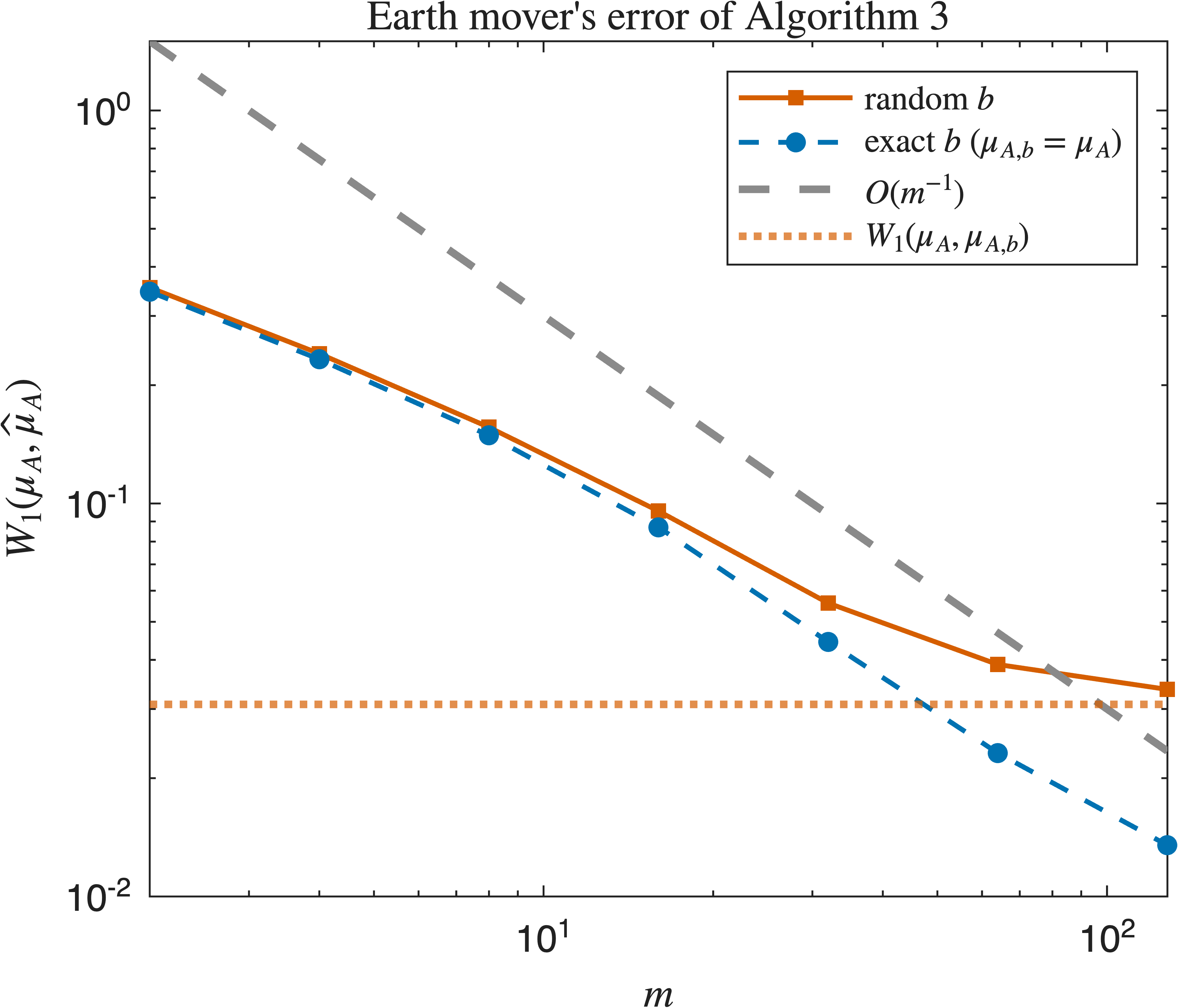} 
    \caption{Error plot for \Cref{a2}.
    This loglog plot depicts the error in the earth mover's metric measured as a function of the maximal degree parameter $m$ in \Cref{a15} for $m = 2, 4, 8, \dots, 128$. The input matrix has dimension $n = 1000$ with eigenvalues forming a circle plus three blobs, where 55\% of the eigenvalues are equispaced on a circle centered at the origin of radius 0.95, and the remaining eigenvalues are IID samples from complex Gaussians with standard deviation $0.1$ and means $-0.5, -0.5i$, and $0.6+0.3i$, respectively. Red-squares depicts the error incurred when using a random starting vector as recommended by the full algorithm \Cref{a2}; for reference, we have included blue-circles, which is the same algorithm run with the ``true'' starting vector which is uniform on all eigenvalues of $A$. Observe that the $O(m^{-1})$ convergence depicted with gray-dashes occurs for both initially, although the random starting vector plateaus at $\wass (\mu_A, \mu_{A,b})$ since the error incurred by using a random starting vector cannot be overcome by increasing the number of iterates $m$. 
    }
    \label{a8}
\end{figure}

\section{Lower bounds for spectral density estimation}\label{a38}

\newcommand{\valg}{v^{\textnormal{alg}}}
\newcommand{\tot}{^{(t)}}

In this section, we prove a lower bound for all matrix-vector query algorithms. That is, we consider any algorithm interacts with the input matrix $A$ only by constructing query vectors $x$ and receiving $Ax$ and $A^*x$. We show that for large matrices, \Cref{a2} makes the asymptotically lowest possible number of queries to achieve its error guarantee stated in \Cref{a1}. In fact, we show the lower bound holds even for algorithms only required to succeed when $A$ is real symmetric.

The main result for this purpose, \Cref{a7}, is a characterization of all adaptive matrix-vector query algorithms which make a small number of queries relative to the matrix size. In particular, closely following ideas of Chewi et al \cite{b26}, we show that any adaptive algorithm which makes at most $m$ matvec queries to large $A$ learns only a small amount of information about $A$ beyond possibly the values of $\tr(A^j)$ for $j=1,\ldots,4m$. We believe this result may be independently useful for producing related lower bounds.

The second step of our lower bound, \Cref{a39}, is to show the values of $\tr(A^j)$ for $j=1,\ldots,4m$ is not enough information about $A$ to reconstruct its spectral density to $o(1/m)$ error. In particular, we produce a family of matrices with far spectral densities, but matching values of $\tr(A^j)$.

In order to argue that matvec algorithms do not learn much information beyond $\tr(A^j)$ for $j=1,\ldots,4m$, we show that the output of any matvec algorithm can be simulated with high probability by an algorithm which takes as input just the values of $\tr(A^j)$ for $j\le 4m$.
To make this formal, we need the notion of \textit{total variation distance,}
\[\tvdist(X,Y)=\inf_{(X,Y)\sim\gamma}\Pr\pare{X\neq Y},\]where the infimum is over all joint distributions $\gamma$ whose marginals are the laws of $X$ and $Y$.
From this formulation of total variation distance, it is clear for every $f$ that\eq{\label{a40}\tvdist(f(X),f(Y))\le\tvdist(X,Y).}
The notation $X\overset d=Y$ means $X$ and $Y$ are equal in distribution, or equivalently $\tvdist(X,Y)=0$.
In this section, $\otimes$ denotes the Kronecker product, $\mathbf1_p\in\R^p$ the column vector of 1s, $I_m$ the $m\times m$ identity, and $e_j$ the $j$th elementary basis column vector of the appropriate dimension.

\begin{proposition}\label{a7}
Let $\valg_0\in\R^n$ be any fixed vector and let \(\valg_j:(\R^n)^j\to\R^n\) be measurable functions. Fix any real symmetric matrix $B\in\R^{p\times p}$ where $p$ is some divisor of $n$ and let $A=U(I_{n/p}\otimes B)U^\top$ be the random matrix where $U$ is a Haar-distributed random orthogonal matrix. If
\[
v_0=\valg_0,\quad
v_j=\valg_j\pare{Av_0,Av_1,\ldots,Av_{j-1}}.\]
then for every \(f_{\textnormal{alg}}:(\R^n)^m\to\mathcal X\) (for any output space $\mathcal X$),
there is a measurable function
\(f_{\textnormal{sim}}:\R^{4m}\to\mathcal X\)
such that
\[
\tvdist\pare{
f_{\textnormal{alg}}\pare{ Av_0,\ldots,Av_{m-1} },
f_{\textnormal{sim}}(\tr(A),\tr(A^2),\cdots,\tr(A^{4m}))}=O\pare{\frac{p^{3/2}m^{3/2}}{n^{1/2}}}.
\]
\end{proposition}

\subsection{Proof of \secthm{a6} assuming \secprop{a7}}
In the following lemma we give a family of matrices with matching traces of powers but separated spectral densities.
\begin{lemma}\label{a39}
    Let $B_m\in\R^{m\times m}$ be the diagonal matrix with $(B_m)_{jj}=\cos\pare{\frac{2j-1}{2m}\pi}$.
    Let $k$, $k'$, $m$, $m'$, $n$ be such that $km=k'm' = n$.
    Then $\tr((B_{m}\otimes I_{k})^i)=\tr((B_{m'}\otimes I_{k'})^i)$ for $i\le\min(m,m')$ but if $m\neq m'$ then
    \[\wass(\mu_{B_{m}\otimes I_{k}},\mu_{B_{m'}\otimes I_{k'}})=\Omega\pare{\frac1{\min(m,m')}}.\]
\end{lemma}\begin{proof}
     By Chebyshev quadrature, one has $\frac1m\tr\pare{B_m^i}=\frac1\pi\int_{-1}^1\frac{x^i}{\sqrt{1-x^2}} \wrt x$ for $i\le m$. In particular,
     \[
    \tr\pare{B_{km}^i}=\frac{km}\pi\int_{-1}^1\frac{x^i}{\sqrt{1-x^2}}\wrt x
     =k\tr\pare{B_m^i}=\tr\pare{(B_m\otimes I_k)^i}.\]
     Since $km=k'm'$, we thus have $\tr\pare{(B_m\otimes I_k)^i}=\tr\pare{(B_{m'}\otimes I_{k'})^i}$ for $i\le\min(m,m')$.
     On the other hand, the earth mover's distance between the spectral measures of $B_{m}\otimes I_{k}$ and $B_{m'}\otimes I_{k'}$ is the minimum $\ell_1$ cost of a coupling between the diagonal entries of those two matrices.
     Specifically, denoting $q_r(x)=\frac{\floor{rx}+\frac12}r$ and replacing $x=j/n$, we have
     \spliteq{\label{a41}}{
     W\equiv\wass(\mu_{B_{m}\otimes I_{k}},\mu_{B_{m'}\otimes I_{k'}})
     &=
     \frac1{n}\sum_{j=1}^{n}\abs{(B_m\otimes I_k)_{jj}-(B_{m'}\otimes I_{k'})_{jj}}
     \\&=\frac1{n}\sum_{j=0}^{n-1}\abs{\cos\pare{\frac{2\floor{j/k}+1}{2m}\pi}-\cos\pare{\frac{2\floor{j/k'}+1}{2m'}\pi}}
     \\&=\int_0^1\abs{\cos(\pi q_m(x))-\cos(\pi q_{m'}(x))}\wrt x
     \\&= \pi \int_0^1\abs{\int_{q_{m'}(x)}^{q_m(x)}\sin(\pi t)\wrt t}\wrt x
     }
Let $S$ be the set $$S=\set{(x,t)\in[0,1]^2: q_m(x)\le t\le q_{m'}(x)\text{ or } q_{m'}(x)\le t\le q_{m}(x)}$$
so that $W$ is simply the integral of $\sin(\pi t)$ over $S$. Let
\[w(t)=\max\set{x:(x,t)\in S}-\min\set{x:(x,t)\in S}.\]
so that \cref{a41} reduces to
\eq{\label{a42}W=\pi\int_0^1w(t)\sin(\pi t)\wrt t.}
Since $\abs{q_r(x)-x}\le\frac1{2r}$, we have $w(t)\le\frac1m+\frac1{m'}$.
Suppose $\int_0^1w(t)\ge c\pare{\frac1m+\frac1{m'}}$ for some positive constant $c$. Then \cref{a42} would be minimized for $w(t)=\frac1m+\frac1{m'}$ for $t\le c$ and $w(t)=0$ for $t>c$. In particular,
\[
W
\ge \pi \pare{\frac1m+\frac1{m'}}\int_0^c\sin(\pi t)\wrt t
\ge\pare{\frac1m+\frac1{m'}}\pare{1-\cos(\pi c)}.\]
It therefore suffices to prove this supposition. To this end, note
\spliteq{}{
\int_0^1w(t)\wrt t
  &=\int_0^1\abs{q_m(x)-q_{m'}(x)}\wrt x
\\&=\frac1n\int_0^1\abs{k\floor{mx}-k'\floor{m'x}+\frac{k-k'}2}\wrt x
\\&=\frac1{n^2}\sum_{j=0}^{n-1}\abs{k\floor{j/k}-k'\floor{j/k'}+\frac{k-k'}2}\wrt x
\\&=\frac1{n^2}\sum_{j=0}^{n-1}\abs{(j\mod k)-(j\mod k')+\frac{k-k'}2}\wrt x
}
Let $d=\gcd(k,k')$. Then we group the terms in the sum based on the residue of $j=id+r$ modulo $d$.
\spliteq{}{
\int_0^1w(t)\wrt t
  &=\frac1{n^2}\sum_{r=0}^{d-1}\sum_{i=0}^{n/d-1}
\abs{(j\mod k)-(j\mod k')+\frac{k-k'}2}
\\&=\frac{d^2}{n^2}\sum_{i=0}^{n/d-1}
\abs{(i\mod k/d)-(i\mod k'/d)+\frac{k-k'}{2d}}.}
Notice that the map $i\mapsto(i\mod k/d,i\mod k'/d)$ is a $n/\lcm(k,k')$-to-1 map from $\set{0,\ldots,n/d-1}$ to $\set{0,\ldots,k/d-1}\times\set{0,\ldots,k'/d-1}$. That is,
\[\int_0^1w(t)\wrt t
=\frac{d^2}{n^2}\frac{n}{\lcm(k,k')}\sum_{x=0}^{k/d-1}\sum_{y=0}^{k'/d-1}\abs{x-y+\frac{k-k'}{2d}}
=\Theta\pare{\frac{d^2}{n^2}\cdot\frac{n}{\lcm(k,k')}\cdot\frac{kk'}{d^2}\frac{k+k'}{d}}.\]
Which, since $kk'/d=\lcm(k,k')$, simplifies to $\Theta\pare{ \frac1{m}+\frac1{m'} }$ as desired.

\end{proof}

\begin{proof}[Proof of \Cref{a6}]
Let $B_{(\cdot)}$ be as in \Cref{a39}, and let $C$ be the constant implied by the $\Omega(\cdot)$ of the same lemma. Fix $0 < \eps < (4sC)^{-1}$, with $s \ge 2$. Define $m_1 < m_2 < \cdots < m_s$ via
\[
    m_j = \ceil{\frac{1}{8C\eps}}+j-1\qfor j\in[s].
\]
Define $N = \prod_{j=1}^s m_j$ and let $n$ be any multiple of $N$ with $n\ge10^4(C')^2(2C)^{-6} s^2\eps^{-6}$, where $C'$ is the absolute constant implied by the total variation bound of \Cref{a7}. Observe that 
\[
    m_1 \leq \dfrac{1}{8C\eps} + 1, \quad m_s \leq \frac{3}{8C\eps}, \quad \eps < \frac{1}{2Cm_{s-1}}.
\]
Set
\[A_j = B_{m_j}\otimes I_{n/m_j}\]
so that $A_j$ is an $n\times n$ matrix with a spectral measure matching that of $B_{m_j}$. By \Cref{a39},
\[
\tr(A_j^i)=
\tr(A_{j'}^i)\]
for all $i\le m_1$ and
\[
\wass(\mu_{A_j},\mu_{A_{j'}})>\frac1{Cm_{s-1}}.\]
Let $X_j$ be the result of some algorithm making at most $m_1/4$ matrix vector queries on input $UA_jU^\top$. By \Cref{a7}, the $X_j$ are all within $\frac{C'm_1^{3/2}m_s^{3/2}}{n^{1/2}}$ in total variation distance to the same random variable (produced by $f_{\textnormal{sim}}$), for some absolute constant $C'$.
Thus, since the target accuracy $\eps$ is less than $\frac1{2Cm_{s-1}}$, with probability $1-\frac{C'm_1^{3/2}m_s^{3/2}}{n^{1/2}}$ over the randomness of $U$, the algorithm must fail for all but one of the $s$ inputs.
The algorithm therefore cannot succeed with overall probability more than
\[\frac1s+{\frac{C'm_1^{3/2}m_s^{3/2}}{n^{1/2}}}
\le
\frac1s+
\frac{C'}{(2C)^3n^{1/2}\eps^3}
\le\frac{1.01}s
.\]
Thus, an algorithm succeeding with probability greater than $1.01/s$ must make more than $\frac{m_1}4\ge\frac1{32C\eps}$ matrix-vector products.
\end{proof}

\subsection{Proof of \secprop{a7}}
We now show how to simulate any matvec algorithm with an algorithm taking as input the traces of powers of $A$.
\begin{lemma}[{\cite[Lemma 6.3]{b38}}]\label{a43}
Let $\valg_j$, $v_j$, and $A$ be as in \Cref{a7}.
Let $z_j\sim\normal(0,I)$. Then there is a function $g$ depending only on $\valg$ such that
\[g(\set{A^iz_j}_{i+j\le m})\overset d=\set{Av_j}_{j<m}.\]
\end{lemma}

\begin{lemma}\label{a44}
Let $A$ be as in \Cref{a7} and $z_j\sim\normal(0,I)$. Then there is a function $h$ such that
\[h(\set{z_{j'}^TA^{i}z_j}_{(i,j,j')\in H})\overset d=\set{A^iz_j}_{i+j\le m}\]
where
\[H=\set{(i,j,j'):i+j+j'\le 2m,j,j'\le m}.\]
\end{lemma}
\begin{proof}
Since $A$ is real symmetric, $\set{z_j^TA^iz_{j'}}_{(i,j,j')\in H}$ is equivalent (up to some duplicated entries) to the set of inner products
$\label{a45}\set{(A^{i'}z_{j'})^\top(A^{i}z_j)}_{i+j\le m,i'+j'\le m}$
which $h$ may compute a factorization of.
For two collections of vectors arranged into the columns of matrices $X$ and $Y$, if all the inner products agree, i.e. $X^*X=Y^*Y$, then there exists a orthogonal $U$ such that $X=UY$. Finally note by construction of $A$ that\(\set{UA^iz_j}_{i+j\le m}\overset{d}{=}\set{A^iz_j}_{i+j\le m}\), so applying a random orthogonal matrix to the factorization found by $h$ achieves the desired distribution.
\end{proof}

\Cref{a46} will be used to argue that the bilinear forms $z_{j'}^\top A^iz_j$ appearing in \Cref{a44} look like Gaussian fluctuations around their means.

\begin{lemma}[{\cite[Theorem 1.3]{b39}}]\label{a46}
Let $Z\in\R^{t\times m}$ and $X\in\R^{m\times m}$ have i.i.d. $\normal(0,1)$ entries. Then
\[
\tvdist\pare{Z^\top Z, tI+\sqrt{\frac t2}\pare{X+X^\top}}\le O\pare{\frac{m^{3/2}}{t^{1/2}}}.\]
The matrix $Z^\top Z$ is called a Wishart matrix.
\end{lemma}

\begin{lemma}\label{a47}
Let $Z\in\R^{tp\times m}$ have i.i.d. $\normal(0,1)$ entries. Fix any symmetric $B\in\R^{p\times p}$ and let $A=B\otimes I_t$. Then there is a Gaussian matrix $G$ with correlated entries entirely characterized by $\set{\tr(A^i)}_{i\le2\ell}$ such that
\[\tvdist\pare{\pmat{Z^\top A Z&\cdots&Z^\top A^\ell Z},G}\le O\pare{\frac{p m^{3/2}}{t^{1/2}}}.\]
\end{lemma}
\begin{proof}
Let $\lambda_1,\ldots,\lambda_p$ be the eigenvalues of $B$ and denote the Vandermonde matrix
\[V=\bmat{
\lambda_1&\cdots&\lambda_1^\ell\\\vdots&&\vdots\\
\lambda_p&\cdots&\lambda_p^\ell
}.\]
Let $B=UDU^\top$ be the diagonalization of $B$. Set $\wt Z=(U^\top\otimes I_t)Z$. 
\[
Z^\top A^iZ
=Z^\top(B^i\otimes I_t)Z
=\wt Z^\top(U^\top\otimes I_t)(B^i\otimes I_t)(U\otimes I_t)\wt Z
=\wt Z^\top(D^i\otimes I_t)\wt Z
\]
If one breaks $\wt Z$ into a column of $p$ many $t\times m$ blocks, then this becomes simply
\[
Z^\top A^iZ
=\sum_{j=1}^p\lambda_j^i\wt Z_j^\top\wt Z_j.\]
Rearranging this equality gives
\[
\pmat{Z^\top AZ&\cdots&Z^\top A^\ell Z}
\overset d=
\pmat{\wt Z_1^\top\wt Z_1&\cdots&\wt Z_p^\top \wt Z_p}\pare{V\otimes I_m}.
\]
Notably, by rotational invariance of $Z$, the matrices $\wt Z_j^\top\wt Z_j$ on the right are all independent Wishart matrices. Set
\[
G=\sqbrac{\pare{t\mathbf1_p^\top\otimes I_m}+\sqrt{t/2}\pmat{X_1+X_1^\top&\cdots&X_p+X_p^\top}}\pare{V\otimes I_m}
\]
where $X_1, \ldots, X_p$ are independent each with i.i.d. $\normal(0,1)$ entries
and note that the entries of $G$ are indeed Gaussian random variables as they are affine functions of the entries of the $X_j$.
By \cref{a40} and \Cref{a46}, we have
\spliteq{}{
&\tvdist\pare{
\pmat{Z^\top AZ&\cdots&Z^\top A^\ell Z},G}
\\&\le
\tvdist\pare{
\pmat{\wt Z_1^\top\wt Z_1&\cdots&\wt Z_p^\top\wt Z_p},\pare{t\mathbf1_p^\top\otimes I}+\sqrt{t/2}\pmat{X_1+X_1^\top&\cdots&X_p+X_p^\top}
}
\\&\le
O\pare{\frac{p m^{3/2}}{t^{1/2}}}.
}
Observe that the factor of $p$ arises from union bounding over the $p$ independent blocks.
Let's now compute the first two moments of $G$. Since $\E(X_j)=0$ for each $j$, the first moment is simply
\[\E(G) 
= (t\mathbf1_p^\top\otimes I_m)(V\otimes I_m)
= (t\mathbf1_p^\top V\otimes I_m)
= \pare{\pmat{\tr(A)&\cdots&\tr(A^\ell)}\otimes I_m}\]
so is determined by $\set{\tr(A^i)}_{i\le\ell}$.
The covariance of the entries is given by
\spliteq{}{
\Cov\pare{
(1\otimes e_b^\top)G(e_i\otimes e_a),\,
(1\otimes e_{b'}^\top)G(e_{i'}\otimes e_{a'})
}
&=\frac t2
\sum_{j=1}^p
\sum_{j'=1}^p
\lambda_j^i
\lambda_{j'}^{i'}\E\pare{e_b^\top(X_j+X_j^\top)e_ae_{b'}^\top(X_{j'}+X_{j'}^\top)e_{a'}}
\\&=\frac t2
\sum_{j=1}^p
\lambda_j^{i+i'}
\E\pare{e_b^\top(X_j+X_j^\top)e_ae_{b'}^\top(X_{j}+X_{j}^\top)e_{a'}}
\\&=\frac 12\tr(A^{i+i'})\E\pare{e_b^\top(X_1+X_1^\top)e_ae_{b'}^\top(X_1+X_1^\top)e_{a'}}}
so is determined by $\set{\tr(A^{i})}_{i\le2\ell}$.
\end{proof}

\begin{proof}[Proof of \Cref{a7}]
We describe the operations performed by $f_{\textnormal{sim}}$. First, because it is given $\tr(A^i)$ for $i\le 4m$, it is able to sample a Gaussian matrix $G$ which, by \Cref{a47} for $\ell=2m$ and $t=n/p$ and the definition of total variation distance, can be written as the mixture
\[
G=\begin{cases}
\pmat{Z^\top AZ&\cdots&Z^\top A^{2m} Z} & \text{w.p. } 1-O\pare{\frac{p^{3/2}m^{3/2}}{n^{1/2}}}\\
\star & \text{o.w.}
\end{cases}\]
where $Z\in\R^{n\times m}$ has i.i.d. $\normal(0,1)$ entries and $\star$ is some unspecified random variable. The blocks $Z^\top A^iZ$ contain the inner products $\set{z_j^\top A^iz_{j'}}_{(i,j,j')\in H}$ where $H$ is the index set in \Cref{a44}. By \Cref{a43} and \Cref{a44} respectively, there are $g$ and $h$ such that
\[
g\pare{h\pare{\set{z_j^\top A^iz_{j'}}_{(i,j,j')\in H}}}
\overset d=g\pare{\set{A^iz_j}_{i+j\le m}}
\overset d=\set{Av_j}_{j<m}.
\]
From there, we may simply apply $f_{\textnormal{alg}}$. The result follows by \cref{a40}.
\end{proof}

\section{Conclusion}\label{a48}

In this work, we studied the problem of estimating the spectral density of a large normal matrix using only matrix-vector products with $A$ and $A^*$. Our main conclusion is that  the spectral density of a normal matrix can be estimated with roughly the same complexity as a Hermitian matrix. More specifically, we showed that when $n \gg 1/\epsilon^2$, for any normal matrix $A$ with $\norm{A}_2\leq 1$, it is sufficient to use $O(1/\varepsilon)$ matrix-vector products in order to return with high probability an $O(\varepsilon)$-error spectral density estimation in earth mover's distance. Moreover, we proved that this rate is asymptotically optimal, even for Hermitian input matrices, in that any algorithm with such a guarantee must in general make $\Omega(1/\varepsilon)$ matrix-vector queries for $n \gg 1/\epsilon^6$.

An important direction for future research is to determine whether our arguments extend beyond the normal matrix setting. Our algorithm and analysis relies crucially on the spectral theorem and on the commutativity of $A$ and $A^*$ or the Hermitian and skew-Hermitian parts of $A$. For non-normal matrices, the matrix will not be unitarily diagonalizable and these commutativity relations break down. Moreover, the eigenvalues of $A$ may not be well-conditioned. It would therefore be interesting either to identify useful classes of nearly normal matrices for which analogues of our results still hold or to prove hardness results showing that the normal case is a genuine boundary of tractability. In our view, understanding which parts of non-Hermitian spectral computation are rendered intractable by non-normality remains one of the central problems of numerical linear algebra. 

\section*{Acknowledgments}
This material is based upon work supported by the National Science Foundation under grant nos. DMS-2513687, AF-2427362,  AF-2427363, AF-2046235, and AF-2045590. Nicholas thanks MathWorks for its support through the MathWorks Fellowship. We would like to thank Rajarshi Bhattacharjee for helpful conversations.

\bibliographystyle{alpha}
\bibliography{outbib}

\newcommand{\etalchar}[1]{$^{#1}$}
\begin{thebibliography}{CdDPL{\etalchar{+}}24}

\bibitem[APJ{\etalchar{+}}18]{b5}
Ryan Adams, Jeffrey Pennington, Matthew Johnson, Jamie Smith, Yaniv Ovadia,
  Brian Patton, and James Saunderson.
\newblock Estimating the spectral density of large implicit matrices.
\newblock {\em \arXiv{1802.03451}}, 2018.

\bibitem[BCG23]{b35}
Sergey Bobkov, Gennadiy Chistyakov, and Friedrich G\"{o}tze.
\newblock {\em Concentration and Gaussian Approximation for Randomized Sums}.
\newblock Springer Nature Switzerland, 2023.

\bibitem[BGVKS22]{b9}
Jess Banks, Jorge Garza-Vargas, Archit Kulkarni, and Nikhil Srivastava.
\newblock Pseudospectral shattering, the sign function, and diagonalization in
  nearly matrix multiplication time.
\newblock {\em Foundations of Computational Mathematics}, 23(6):1959--2047,
  2022.

\bibitem[BHOT24]{b36}
Nicolas Boull{\'e}, Diana Halikias, Samuel~E Otto, and Alex Townsend.
\newblock Operator learning without the adjoint.
\newblock {\em Journal of Machine Learning Research}, 25(364), 2024.

\bibitem[BJM{\etalchar{+}}25]{b18}
Rajarshi Bhattacharjee, Rajesh Jayaram, Cameron Musco, Christopher Musco, and
  Archan Ray.
\newblock Improved spectral density estimation via explicit and implicit
  deflation.
\newblock In {\em \SODA{2025}}, 2025.

\bibitem[BKM22]{b17}
Vladimir Braverman, Aditya Krishnan, and Christopher Musco.
\newblock Sublinear time spectral density estimation.
\newblock In {\em \STOC{2022}}, 2022.

\bibitem[BN23]{b38}
Ainesh Bakshi and Shyam Narayanan.
\newblock {Krylov} methods are (nearly) optimal for low-rank approximation.
\newblock In {\em \FOCS{2023}}, 2023.

\bibitem[Cai16]{b22}
Alberto Cairo.
\newblock Download the datasaurus: Never trust summary statistics alone; always
  visualize your data.
\newblock
  \url{https://web.archive.org/web/20160901105257/http://www.thefunctionalart.com/2016/08/download-datasaurus-never-trust-summary.html},
  2016.

\bibitem[CdDPL{\etalchar{+}}24]{b26}
Sinho Chewi, Jaume de~Dios~Pont, Jerry Li, Chen Lu, and Shyam Narayanan.
\newblock Query lower bounds for log-concave sampling.
\newblock {\em Journal of the ACM}, 71(4), 2024.

\bibitem[CTU21]{b16}
Tyler Chen, Thomas Trogdon, and Shashanka Ubaru.
\newblock Analysis of stochastic {L}anczos quadrature for spectrum
  approximation.
\newblock In {\em \ICML{2021}}, 2021.

\bibitem[CU24]{b21}
Cecilia Chen and John Urschel.
\newblock Estimating the numerical range with a {K}rylov subspace.
\newblock {\em \arXiv{2411.19165}}, 2024.

\bibitem[CY92]{b6}
Ling-Lie Chau and Yue Yu.
\newblock Unitary polynomials in normal matrix models and wave functions for
  the fractional quantum {Hall} effects.
\newblock {\em Physics Letters A}, 167(5-6), 1992.

\bibitem[CZ98]{b7}
Ling-Lie Chau and Oleg Zaboronsky.
\newblock On the structure of correlation functions in the normal matrix model.
\newblock {\em Communications in mathematical physics}, 196(1), 1998.

\bibitem[DMN25]{b37}
Francesco Dell'Accio, Francisco Marcell{\'a}n, and Federico Nudo.
\newblock An interpolation--regression approach for function approximation on
  the disk and its application to cubature formulas.
\newblock {\em Advances in Computational Mathematics}, 51(6), 2025.

\bibitem[FP57]{b20}
Ky~Fan and Gordon Pall.
\newblock Imbedding conditions for {Hermitian} and normal matrices.
\newblock {\em Canadian Journal of Mathematics}, 9, 1957.

\bibitem[GKX19]{b2}
Behrooz Ghorbani, Shankar Krishnan, and Ying Xiao.
\newblock {An investigation into neural net optimization via {H}essian
  eigenvalue density}.
\newblock In {\em International Conference on Machine Learning}, pages
  2232--2241. PMLR, 2019.

\bibitem[HMT11]{b12}
Nathan Halko, Per-Gunnar Martinsson, and Joel~A Tropp.
\newblock Finding structure with randomness: Probabilistic algorithms for
  constructing approximate matrix decompositions.
\newblock {\em SIAM review}, 53(2), 2011.

\bibitem[Jac12]{b31}
Dunham Jackson.
\newblock On approximation by trigonometric sums and polynomials.
\newblock {\em Transactions of the American Mathematical society}, 13(4), 1912.

\bibitem[JP94]{b0}
Janez Jaklic and Peter Prelovsek.
\newblock {Lanczos} method for the calculation of finite-temperature quantities
  in correlated systems.
\newblock {\em Physical Review B}, 49(7), 1994.

\bibitem[KW92]{b27}
Jacek Kuczy{\'n}ski and Henryk Wo{\'z}niakowski.
\newblock Estimating the largest eigenvalue by the power and lanczos algorithms
  with a random start.
\newblock {\em SIAM journal on matrix analysis and applications},
  13(4):1094--1122, 1992.

\bibitem[Lan50]{b10}
Cornelius Lanczos.
\newblock An iteration method for the solution of the eigenvalue problem of
  linear differential and integral operators.
\newblock {\em Journal of Research of the National Bureau of Standards}, 45(4),
  1950.

\bibitem[LSY16]{b1}
Lin Lin, Yousef Saad, and Chao Yang.
\newblock Approximating spectral densities of large matrices.
\newblock {\em SIAM review}, 58(1), 2016.

\bibitem[LXES19]{b3}
Ruipeng Li, Yuanzhe Xi, Lucas Erlandson, and Yousef Saad.
\newblock The eigenvalues slicing library {(EVSL)}: Algorithms, implementation,
  and software.
\newblock {\em SIAM Journal on Scientific Computing}, 41(4), 2019.

\bibitem[Mey26]{b39}
Raphael~Arkady Meyer.
\newblock A non-asymptotic bound on the tv distance between a wishart matrix
  and an appropriately scaled goe matrix.
\newblock {\em arXiv preprint arXiv:2606.16018}, 2026.

\bibitem[MF17]{b23}
Justin Matejka and George Fitzmaurice.
\newblock Same stats, different graphs: Generating datasets with varied
  appearance and identical statistics through simulated annealing.
\newblock In {\em Proceedings of the 2017 CHI Conference on Human Factors in
  Computing Systems}, pages 1290--1294, 2017.

\bibitem[MM15]{b14}
Cameron Musco and Christopher Musco.
\newblock Randomized block {K}rylov methods for stronger and faster approximate
  singular value decomposition.
\newblock {\em \NIPS{2015}}, 28, 2015.

\bibitem[MMMW21]{b24}
Raphael~A. Meyer, Cameron Musco, Christopher Musco, and David~P. Woodruff.
\newblock Hutch++: Optimal stochastic trace estimation.
\newblock In {\em \SOSA{2021}}, 2021.

\bibitem[MMRS25]{b19}
Cameron Musco, Christopher Musco, Lucas Rosenblatt, and Apoorv~Vikram Singh.
\newblock Sharper bounds for {C}hebyshev moment matching, with applications.
\newblock {\em \COLT{2025}}, 2025.

\bibitem[NS64]{b34}
D.~J. Newman and H.S. Shapiro.
\newblock {Jackson's} theorem in higher dimensions.
\newblock In {\em On Approximation Theory / {\"U}ber Approximationstheorie:
  Proceedings of the Conference held in the Mathematical Research Institute at
  Oberwolfach}, pages 208--219, 1964.

\bibitem[RB25]{b29}
I.O. Raikov and Y.M. Beltukov.
\newblock The kernel polynomial method based on {J}acobi polynomials.
\newblock {\em Applied Mathematics and Computation}, 490, 2025.

\bibitem[Riv81]{b30}
Theodore~J. Rivlin.
\newblock {\em An introduction to the approximation of functions}.
\newblock Courier Corporation, 1981.

\bibitem[Rys63]{b8}
Herbert Ryser.
\newblock {\em Combinatorial Mathematics}.
\newblock American Mathematical Society, 1963.

\bibitem[Saa11]{b11}
Yousef Saad.
\newblock {\em Numerical methods for large eigenvalue problems: revised
  edition}.
\newblock SIAM, 2011.

\bibitem[SUW26]{b32}
Rikhav Shah, John Urschel, and Nicholas West.
\newblock What is jackson's constant?
\newblock {\em arXiv preprint arXiv:2607.23375}, 2026.

\bibitem[SW23]{b15}
William Swartworth and David~P. Woodruff.
\newblock Optimal eigenvalue approximation via sketching.
\newblock In {\em \STOC{2023}}, 2023.

\bibitem[Tim63]{b33}
Aleksandr~F. Timan.
\newblock {\em Theory of Approximation of Functions of a Real Variable}.
\newblock International Series of Monographs in Pure and Applied Mathematics.
  Pergamon Press, 1963.

\bibitem[Urs21]{b28}
John~C. Urschel.
\newblock Uniform error estimates for the lanczos method.
\newblock {\em SIAM Journal on Matrix Analysis and Applications}, 42(3), 2021.

\bibitem[Woo14]{b13}
David~P. Woodruff.
\newblock Sketching as a tool for numerical linear algebra.
\newblock {\em Foundations and Trends{\textregistered} in Theoretical Computer
  Science}, 10(1-2), 2014.

\bibitem[WWAF06]{b4}
Alexander Wei{\ss}e, Gerhard Wellein, Andreas Alvermann, and Holger Fehske.
\newblock The kernel polynomial method.
\newblock {\em Reviews of modern physics}, 78(1), 2006.

\bibitem[WZZ22]{b25}
David~P. Woodruff, Fred Zhang, and Richard Zhang.
\newblock Optimal query complexities for dynamic trace estimation.
\newblock In {\em \NIPS{2022}}, 2022.

\end{thebibliography}

\end{document}